\documentclass[11pt]{article}

\usepackage{amsmath,amsthm,verbatim,amssymb,amsfonts,amscd, graphicx, mathrsfs, mathtools, hyperref, cleveref, dsfont, nomencl, xcolor, soul}
\usepackage{graphics}
\usepackage{tikz}
\usepackage{tikz-cd}
\usepackage{etoolbox, textcomp}
\usepackage{cancel}
\usetikzlibrary{matrix,arrows,decorations.pathmorphing}
\setstcolor{red}

\theoremstyle{plain}
\newtheorem{theorem}{Theorem}
\newtheorem{corollary}[theorem]{Corollary}
\newtheorem{lemma}[theorem]{Lemma}
\newtheorem{proposition}[theorem]{Proposition}
\newtheorem{conjecture}[theorem]{Conjecture} 
\newtheorem{remark}[theorem]{Remark} 
 
\theoremstyle{definition}
\newtheorem{definition}[theorem]{Definition}

\definecolor{darkgreen}{rgb}{0,0.5,0}
\definecolor{darkblue}{rgb}{0,0.3,0.7}

\hypersetup{
    colorlinks=true, % make the links colored
    linkcolor=darkblue, % color TOC links in blue
    urlcolor=blue, % color URLs in red
    linktoc=all, % 'all' will create links for everything in the TOC
    citecolor=darkgreen
    }

\newcommand\doi[2]        {\href{http://dx.doi.org/#1}{#2}}

\newcommand{\sfw}{\mathsf{w}}
\newcommand{\Ap}{\mathcal{A}_p}

\newcommand{\Bp}{\mathcal{B}_p}
\newcommand{\Mp}{\mathcal{M}(p)}
\newcommand{\Wp}{\mathcal{W}(p)}
\newcommand{\ind}{\mathscr{F}}
\newcommand{\C}{\mathcal{C}}

\newcommand{\ZZ}{\mathbb{Z}}
\newcommand{\QQ}{\mathbb{Q}}
\newcommand{\CC}{\mathbb{C}}
\newcommand{\RR}{\mathbb{R}}
\newcommand{\CCdot}{\ddot{\mathbb{C}}}

\newcommand{\Heis}{\mathsf{H}}
\newcommand{\Hlie}{\widehat{\mathfrak{h}}}

\newcommand{\sS}{\mathsf{S}}
\newcommand{\mdot}{\cdot}
\newcommand{\voa}{vertex operator algebra}
\newcommand{\voas}{vertex operator algebras}
\newcommand{\Ck}{\mathbb{C}_{\ell p}^H}
\newcommand{\HR}{\mathcal{H}^{\oplus}_{i\mathbb{R}}}

\newcommand{\U}[1]{U_q(\mathfrak{sl}_{#1})}
\newcommand{\UH}[1]{U_q^H(\mathfrak{sl}_{#1})}
\newcommand{\UHbar}[1]{\overline{U}_q^H(\mathfrak{sl}_{#1})}
\newcommand{\FockH}[1]{\mathrm{F}_{#1}}
\newcommand{\FockC}[1]{\mathsf{F}_{#1}}
\newcommand{\FockS}[1]{F_{#1}}

\newcommand{\FC}[2]{\mathrm{F}_{#1}\boxtimes \CC_{#2}^H}
\newcommand{\FV}[2]{\mathrm{F}_{#1} \boxtimes V_{#2}}
\newcommand{\SC}[2]{S_{#1} \otimes \mathbb{C}^H_{#2}}
\newcommand{\PC}[2]{P_{#1} \otimes \mathbb{C}^H_{#2}}
\newcommand{\FSC}[3]{\mathrm{F}_{#1} \boxtimes (S_{#2} \otimes \mathbb{C}_{#3}^H)}
\newcommand{\FPC}[3]{\mathrm{F}_{#1} \boxtimes (P_{#2} \otimes \mathbb{C}_{#3}^H)}

\newcommand{\ch}[2]{\mathrm{ch}[#1](#2)}
\newcommand{\sch}[2]{\mathrm{sch}[#1](#2)}

\newcommand{\sln}[1]{\mathfrak{sl}_{#1}}
\newcommand{\module}[3]{\mathsf{M}_{#1, #2, #3}}
\newcommand{\hopf}[6]{\sS^{\hopflink}_{(#1,#2,#3),(#4,#5,#6)}}
\newcommand{\sltwo}{\mathfrak{sl}_2}
\newcommand{\simTypiBpNAME}[2]{E^V_{#1,\,#2}} 

\newcommand{\simAtypiBpNAME}[3]{E^S_{#1,\,#2,\,#3}} 

\newcommand{\projVBpNAME}[2]{Q^\text{\em{V}}_{#1,\,#2}} 

\newcommand{\projPBpNAME}[3]{Q^\text{\em{P}}_{#1,\,#2,\,#3}} 

\newcommand{\hopflink}{{\,\text{\textmarried}}}

\DeclareMathOperator{\Id}{Id} % identité
\DeclareMathOperator{\Hom}{Hom} % homomorphismes
\DeclareMathOperator{\qdim}{qdim} % qdim
\DeclareMathOperator{\rep}{Rep}
\DeclareMathOperator{\Rep}{Rep}
\DeclareMathOperator{\tr}{tr}

\allowdisplaybreaks
\numberwithin{equation}{section}

\title{Braided Tensor Categories related to $\Bp$ Vertex Algebras}
\author{Jean Auger, Thomas Creutzig, Shashank Kanade, Matthew Rupert}

\date{}

\begin{document}
\maketitle

\begin{abstract}
The $\Bp$-algebras are a family of \voas{} parameterized by $p\in \mathbb Z_{\geq 2}$. They are important examples of logarithmic CFTs and appear  as chiral algebras of type $(A_1, A_{2p-3})$ Argyres-Douglas theories. The first member of this series, the $\mathcal B_2$-algebra, are the well-known symplectic bosons also often called the $\beta\gamma$ \voa.

We study categories related to the $\Bp$ \voas{} using their conjectural relation to unrolled restricted quantum groups of $\sltwo$.
These categories are braided, rigid and non semi-simple tensor categories. We list their simple and projective objects, their tensor products and their Hopf links. The latter are succesfully compared to modular data of characters thus confirming a proposed Verlinde formula of David Ridout and the second author. 
 \end{abstract}

\vspace{1cm}
%\tableofcontents
\vspace{1cm}

%%%
\section{Introduction}

In the present work we study a family of braided tensor categories associated to the $\Bp$ vertex algebras. Our motivation is fourfold: These \voas{} serve as prototypical examples for logarithmic two-dimensional conformal field theories and appear in the context of four-dimensional super conformal field theories. Moreover, they nicely illustrate the power of the theory of \voa{} extensions and they allow to test our ideas of Verlinde's formulae for logarithmic conformal field theories. 

Let us start by describing these four aspects and the underlying principle of correspondences between representation categories of quantum groups and \voas. 

\subsection{Quantum groups and \voas}

There are numerous
 relations between representation categories of \voas{} associated to a simple Lie algebra $\mathfrak g$ and module categories of quantum groups of type $\mathfrak g$. The most famous instance of this is of course Kazhdan-Lusztig's correspondence \cite{KL1, KL2, KL3, KL4} stating a braided equivalence between representation categories of affine Lie algebras at generic level and quantum groups at the corresponding $q$-parameter. The relevant representation categories of affine Lie algebras have then been understood as a vertex tensor categories of the corresponding affine \voas{} \cite{Fi, Zha, H-applicability}; see \cite{H-survey} for a review of the state of the art. For us certain \voas{} that are realized as kernels of so-called screening charges are of interest. These screening charges are related to simple roots of a simple Lie algebra and the corresponding \voas{} are automatically extensions of principal W-algebras associated to the Lie algebra. See \cite{FT} for general constructions. Except for \cite{CM2} only the example of $\mathfrak g= \mathfrak{sl}_2$ has been studied. In that latter case the corresponding \voas{} are called triplet $\Wp$ and singlet algebras $\Mp$, parameterized by $p\in\mathbb Z_{\geq 2}$, with the singlet being a subalgebra of the triplet. 
Studying  representation categories of \voas{} is typically challenging (see however \cite{TW, AM2, AM3, FGST, FGST2} for the triplet and \cite{A, AM1, CM1, CMR} for the singlet) therefore, a conjectural correspondence to a representation categories of a quantum groups is welcome. 
Let $q=e^{\pi i /p}$ then the quasi Hopf algebra \cite{CGR} corresponding to $\Wp$ is very closely related to the restricted quantum group of $\sltwo$, $\overline{U}_q(\sltwo)$, while the one corresponding to the singlet $\Mp$ is the unrolled restricted quantum group $\UHbar{2}$ of \cite{CGP}. Our basic conjecture is an equivalence of braided tensor categories for the unrolled restricted quantum group and singlet \voa. This conjecture appeared first in \cite{CGP} and has been refined and tested in \cite{CM1, CMR, CGR, R}, see also \cite{L, FL} for higher rank studies. 

\begin{conjecture}\label{keyconj}\textup{\cite[Conjecture 5.8]{CGR}}

The category of weight modules of the unrolled restricted quantum group $\UHbar{2}$ at $q=e^{\pi i/p}$ is braided equivalent to the smallest category of modules of the singlet \voa{} $\Mp$ that  contains all simple modules and which is complete with respect to taking tensor products, finite sums and subquotients.

\end{conjecture}

\subsection{Logarithmic conformal field theory and Verlinde's formula}

A \voa{} is the symmetry/chiral algebra of a two-dimensional conformal field theory. A non semi-simple action of the Virasoro zero-mode results in possibly logarithmic singularities in correlation functions. This is the reason for calling conformal field theories based on \voas{} with non semi-simple representation categories logarithmic CFTs. We refer to \cite{CR1} for an introduciton. 

In contrast, a conformal field theory based on a \voa{} whose representation category is a finite, semi-simple, rigid, braided and factorizable tensor category is called rational. Such a category is usually called a modular tensor category and the corresponding \voa{} is called strongly rational. 
In this case, the physics axioms of conformal field theory say that Verlinde's formula should hold \cite{V, MS}. For this, let $M_0, M_1, \dots, M_n$ denote the inequivalent simple objects of a strongly rational \voa{}  $V=M_0$. Then torus one-point functions are the graded traces
\begin{align}
\text{ch}[M_i](v, \tau) = \text{tr}_{M_i}\left( o(v) q^{L_0- \frac{c}{24}}\right), \qquad q=e^{2\pi i \tau}
\end{align}
with $o(v)\in \text{End}(M_i)$ the zero-mode corresponding to $v\in V$ and $c$ the central charge of the \voa. 
They carry an action of the modular group \cite{Z} and especially the modular $\sS$-transformation defines the modular $\sS$-matrix,
\begin{align}
\text{ch}[M_i]\left(v, -\frac{1}{\tau}\right) =  \tau^k \sum_{j=0}^n  \sS^\chi_{i, j} \text{ch}[M_j](v, \tau)
\end{align}
with $k$ the modified degree of $v$ as introduced by \cite{Z}. We do not go into more details here but refer to the introduction of \cite{CG} for a more precise explanation. 
Let ${N_{ij}}^k$ be the fusion rules of the theory, i.e. 
\begin{align}
M_i \boxtimes_V M_j \cong \bigoplus_{k=0}^n {N_{ij}}^k \, M_k
\end{align}
then Verlinde's formula says that
\begin{equation} \label{ratVer}
{N_{ij}}^k = \sum_{\ell =0}^n \frac{\sS^\chi_{i,\ell}\sS^\chi_{j,\ell}(\sS^\chi)^{-1}_{\ell, k}}{\sS^\chi_{0,\ell}}.
\end{equation}
This formula follows directly from 
\begin{equation}\label{trueVer}
\frac{\sS^\hopflink_{i, j}}{\sS^\hopflink_{0, j}} =\frac{\sS^\chi_{i, j}}{\sS^\chi_{0, j}} 
\end{equation}
with the Hopf links $\sS^\hopflink_{i, j} = \text{tr}_{M_i \boxtimes_V M_j}\left( c_{j, i} \circ c_{i, j}\right)$ where the  $c_{i, j}: M_i \boxtimes_V M_j \xrightarrow{\cong} M_j \boxtimes_V M_i$ are the braiding isomorphisms, since in any finite rigid braided tensor category Hopf links satisfy
\begin{align}
\frac{\sS^\hopflink_{i, \ell}}{\sS^\hopflink_{0, \ell}} \frac{\sS^\hopflink_{j, \ell}}{\sS^\hopflink_{0, \ell}} = \sum_{k=0}^n {N_{ij}}^k\frac{\sS^\hopflink_{k, \ell}}{\sS^\hopflink_{0, \ell}} 
\end{align}
if $\ell$ labels a simple object. Then \eqref{ratVer} follows directly from \eqref{trueVer} and invertibility of the Hopf link $\sS^\hopflink$-matrix in modular tensor categories. This famous Theorem is due to Yi-Zhi Huang \cite{H1, H2, H3}.

It is natural to ask for a variant of Verlinde's formula for \voas{} with non semi-simple finite representation categories. The picture that is promoted in \cite{CG} is that there should still hold a relation of the type \eqref{trueVer} with slight modifications: since traces on negligible objects vanish one then needs to replace Hopf links by modified Hopf links, that is modified traces of double braidings; similarly one also needs to take into account so-called pseudo trace functions of modules \cite{Miy}. This picture is verified in examples that are based on conjectural correspondences to restricted quantum groups \cite{CG, CMR}, see also \cite{GR, FGR} for further work on the Verlinde formula in this context. 

In practice, most interesting \voas{} (such as the affine \voas{} at admissible level) have representation categories that are not even finite, they have uncountably many inequivalent simple objects. However, jointly with David Ridout, one of us was able to nonetheless conjecture a Verlinde's formula for affine \voas{} of $\mathfrak{sl}_2$ at admissible level by treating characters as formal distributions \cite{CR2, CR3}, see \cite{RW} for a review. This conjecture is completely open, except for some encouraging computations of fusion rules \cite{Ga, Ri, AP} and a recent proof of a formula of type \eqref{trueVer} for the finite subcategory of ordinary modules for all affine \voas{} of simply-laced Lie algebras at admissible level \cite{CHY, C1}.

The conjectural Verlinde formula for the $\Bp$-algebras is of the same type as admissible level affine $\mathfrak{sl}_2$ and actually $\mathcal B_3$ is nothing but the affine \voa{} of $\mathfrak{sl}_2$ at level $-4/3$. Assuming the quantum group correspondence (Conjecture \ref{keyconj}) to be correct we indeed verify that normalized Hopf links coincide up to complex conjugation with normalized modular $S^\chi$-coefficients. These will be the results of section \ref{Sectiontypical} and \ref{Sectionatypical}, see also Theorem \ref{introcomparison} of the introduction.

\subsection{Chiral algebras for Argyres-Douglas theories}

It turns out that \voas{} play a very important role in studying higher dimensional supersymmetric gauge theories \cite{BL+, CG2} and the \voas{} of the present work appear as chiral algebras of Argyres-Douglas theories of type $(A_1,  A_{2p-3})$ \cite{C}. Argyres-Douglas theories are four-dimensional $\mathcal N = 2$ supersymmetric field theories \cite{ADj, APSW, EHIY}. Schur indices of these theories are believed to coincide with characters of associated \voas{} \cite{BN1,XYY,BN2,CGS,CGS2,C,C3} while indices of line and surface defects are related to characters of modules of the \voas{} \cite{CGS,CGS2}. Finally, modularity and Verlinde's formula also have some nice interpretation in the gauge theory \cite{FPYY,NY,KSY}.
 For some further recent progress see \cite{BLN,SXY,S,BLN2,IS,ASS,FS,Gia,CN,BN3,BL,XY1,XY2}.
 
The $\Bp$-algebra can thus be seen as a prototypical example of \voas{} appearing in such theories and so let us describe their modules in a few words before giving all details in the later part of this work. There are a few types of modules. Firstly ordinary modules: these are modules that have finite-dimensional top level subspace and conformal weight spaces are all finite-dimensional as well. Their characters converge on the upper-half plane and are actually modular. The second type are highest weight modules with infinite-dimensional top level subspace. Their character converges in a certain domain if we consider characters graded by conformal weight and Heisenberg weight (fugacity in physics). Analytic continuation of these characters then form a vector-valued meromorphic Jacobi form of negative index. Finally, there are modules that are not even of highest-weight type. They are called relaxed-highest-weight modules and their characters only converge in a distributional sense. Furthermore there are also modules that are twisted by spectral flow and then there are projective covers of simple modules. Concerning modularity, the physics interest has mainly been in the study of the modules whose characters analytically continue to meromorphic Jacobi forms. They then also have considered a naive Verlinde formula and compared them to fusions of line defects in the gauge theory. See e.g. \cite{NY} for nice discussions. However, negative fusion coefficients appeared! Let us quickly explain how to avoid them. It turns out that the map of modules to analytic continuations of characters is infinity to one. The reason is that there are infinitely many modules whose characters converge in certain different domains but their analytic continuations coincide up to possible signs. As we will explain later, see Theorem \ref{thm:Verintro}, there is the notion of semisimplification of a category and modular properties of these meromorphically continued characters capture fusion of a Grothendieck ring associated to the semisimplification of the category. For this one has to wisely choose modules representing this Grothendieck ring and if we follow the rule that with a module $M$ also its dual $M^*$ has to be a representative of this Grothendieck ring then indeed we get a working Verinde's formula with non-negative integer fusion rules. 
For details see section \ref{Sectionatypical}.

There are other representation categorical questions that are rather important in gauge theories. For example Ext-groups of \voas{} including the $\Bp$-algebras seem to correspond to the algebra of functions on the Coulomb branch of the theory \cite{CG3, CCG}.

\subsection{Vertex algebra extensions}

The present work is a continuation of our previous studies. The starting point is the insight that one can study \voa{} extension in a purely categorical language. Firstly, let $V$ be a \voa{} with a vertex tensor category $\C$ of $V$-modules, i.e. $\C$ is especially a braided tensor category with a system of balancing isomorphisms $\theta$. Then an object $A$ of $\C$ containing $V$ with multiplicity one extends $V$ to a larger \voa{} if and only if $A$ can be given the structure of a commutative algebra object in $\C$ with a trivial $\theta$ \cite{HKL}. This result can be very easily extended to the case of
$\frac{1}{2}\ZZ$-graded vertex operator super-algebras, \cite{CKL}.
Moreover there is a notion of local modules for a commutative algebra object and this category of local objects is equivalent as a braided tensor category to the category of modules of the \voa{} $A$ that lie in $\C$ and furthermore there is a induction functor from the category of those $V$-modules that centralize $A$ to local $A$-modules  \cite{CKM}. In fortunate situations all local modules of interest can be realized via induction and representation categories associated to quite a few logarithmic theories have already successfully been studied in this way, see \cite{ACR, CFK, CKLR2, CLRW}. In addition the Hopf links that play such an important role in Verlinde's formula commute with the induction functor. With this knowledge we can study the $\Bp$ vertex algebras. Firstly, by construction they are infinite order simple current extensions of a singlet \voa{} times a Heisenberg \voa{} \cite{CRW}. A simple current is an invertible object in the tensor category and these give rise to the nicest (or simplest) \voa{} extensions. The theory of simple current extensions in non semi-simple tensor categories is developed in \cite{CKL} and the underlying category is actually a completion as one has to allow for infinite direct sums, see \cite{AR, CGR} for details. The singlet \voa{} is the $U(1)$-orbifold of the probably best-known $C_2$-cofinite yet non-rational \voa{}, the triplet algebra \cite{AM2, FGST,FGST2, TW}. The category of Fock modules of the Heisenberg \voa{} is nothing but $\CC$-graded vector spaces and thus rather trivial so that the interesting structure of $\Bp$-algebra modules will come from objects of the singlet algebra. The representation theory of the latter is somehow known and we refer to section 5 of \cite{CGR} to a summary. The most important point for us is that there is a conjectural equivalence to weight modules of a unrolled restricted quantum group of $\mathfrak{sl}(2)$, see Conjecture \ref{keyconj} above. This conjecture has been tested in various ways \cite{CM1, CMR, CGR, R} and so the braided tensor category that we study is the one of the algebra object corresponding to the $\Bp$-algebra under the conjectural equivalence.

\subsection{Results}

In this section, we summarise the results of the paper. For the detailed notation we refer to the main body of the work.

Let $p$ be a positive integer at least $2$ and let $q=e^{\pi i/p}$.
Define the category $\C:=\HR \boxtimes \UHbar{2}\text{-Mod}$ to be the Deligne product of a category $\HR$ of $i\mathbb{R}$-graded complex vector spaces (see subsection \ref{Heisenberg} for details) and the weight category of the quantum group $\UHbar{2}$. $\C$ contains simple currents $\FC{\lambda_p}{p}$ where $\lambda_p^2:=-p/2$ and it is shown in Section \ref{SectionBp} that the corresponding object 
\begin{align}
\Ap:= \bigoplus\limits_{k \in \mathbb{Z}} (\FC{k\lambda_p}{kp})
\end{align}
in the extended category $\C_{\oplus}$ is a commutative algebra object with unique (up to isomorphism) algebraic structure (see Proposition \ref{Apalgebra}). Hence, one can define the corresponding representation categories of modules and local modules for $\Ap$, $\Rep(\Ap)$ and $\Rep^{0}(\Ap)$, respectively (see \ref{algebracurrents}).
Importantly, there is an induction functor $\ind :\C \rightarrow \Rep(\Ap)$ and since $\Ap$ is a simple current extension we can determine all simple local modules using induction. 
 The categorical structure of $\Rep^{0}(\Ap)$ is determined in Proposition \ref{RepApstructure} and is completely inherited from $\C$. The following theorem lists simple and projective objects and for notation and a description of $\C$ we refer to the background section. 
\begin{theorem} The simple modules in $\Rep^{0}(\Ap)$ are (with $\lambda_p^2:=-p/2$ and  $\ddot{\CC}:=(\CC-\ZZ) \cup p\ZZ$)
\begin{align}
 \simTypiBpNAME{\gamma}{\alpha} &= \ind\left(\FV{\gamma}{\alpha}\right)& \text{with } \alpha &\in \CCdot \text{ and } \gamma \lambda_p + \frac{\alpha+p-1}{2} \in \ZZ \; , \\
 \simAtypiBpNAME{\gamma}{i}{\ell} &= \ind\left(\FSC{\gamma}{i}{\ell p}\right)&  \text{with } i &\in \{0,\dots,p-2\},\ell \in \mathbb{Z} \text{ and } \gamma \lambda_p + \frac{i+p\ell}{2} \in \ZZ \; .
\end{align}
We have families of indecomposable modules:
\begin{align}
\projVBpNAME{\gamma}{\alpha} &= \ind\left(\FV{\gamma}{\alpha}\right) &\text{with } \alpha &\notin \CCdot \text{ and } \gamma \lambda_p + \frac{\alpha+p-1}{2} \in \ZZ \; , \\
\projPBpNAME{\gamma}{i}{\ell} &= \ind\left(\FPC{\gamma}{i}{\ell p}\right) \ &\text{with } i &\in \{0,\dots,p-2\}, \ell \in \mathbb{Z} \text{ and } \; \gamma \lambda_p + \frac{i + \ell p}{2} \in \ZZ,
\end{align}
with $\projPBpNAME{\gamma}{i}{\ell}$ being projective, and the above modules satisfy 
\begin{align}
E^V_{\gamma,\alpha} \cong E^V_{\gamma+k\lambda_p,\alpha+pk},&\quad\quad E^S_{\gamma,i,\ell} \cong E^S_{\gamma+k\lambda_p,i,\ell+k},\\
Q^V_{\gamma,\alpha} \cong Q^V_{\gamma+k\lambda_p,\alpha+pk},&\quad\quad Q^P_{\gamma,i,\ell}  \cong Q^P_{\gamma+k\lambda_p,i,\ell+k},
\end{align}
for all $k \in \mathbb{Z}$. 
Given $\alpha \not \in \ddot{\mathbb{C}}$, we can write $\alpha=i+\ell p$ for some $i=1,...,p-1$ and $\ell \in \mathbb{Z}$. The indecomposable modules admit the following Loewy diagrams:
\begin{center}
\begin{tikzpicture}[baseline=(current bounding box.center)]
\node (top) at (1,1) [] {$\simAtypiBpNAME{\gamma}{(p-1)-i}{\ell}$};
\node (bottom) at (1,-1) [] {$\simAtypiBpNAME{\gamma}{i-1}{\ell+1}$};
\draw[->, thick] (top) -- (bottom);
\node (label) at (-0.75,0) [circle, inner sep=2pt, color=black, fill=brown!25!] {$\projVBpNAME{\gamma}{i+\ell p}$};
\end{tikzpicture}
\hspace{3cm}
\begin{tikzpicture}[baseline=(current bounding box.center)]
\node (top) at (1,1) [] {$\simTypiBpNAME{\gamma}{i+1-p+\ell p}$};
\node (bottom) at (1,-1) [] {$\simTypiBpNAME{\gamma}{p-1-i+\ell p}$};
\draw[->, thick] (top) -- (bottom);
\node (label) at (-0.5,0) [circle, inner sep=2pt, color=black, fill=brown!25!] {$\projPBpNAME{\gamma}{p-1}{\ell}$};
\end{tikzpicture}
\\
\begin{tikzpicture}[baseline=(current bounding box.center)]
\node (tag) at (-6,0) [] {\em{For $i = 1, \dots , p-2 \,$:}};
\node (top) at (0,1.5) [] {$\simAtypiBpNAME{\gamma}{i}{\ell}$};
\node (middleI) at (-2.25,0) [] {$\simAtypiBpNAME{\gamma}{(p-2)-i}{\ell-1}$};
\node (middleII) at (2.25,0) [] {$\simAtypiBpNAME{\gamma}{(p-2)-i}{\ell+1}$};
\node (bottom) at (0,-1.5) [] {$\simAtypiBpNAME{\gamma}{i}{\ell}$};
\draw[->, thick] (top) -- (middleI);
\draw[->, thick] (top) -- (middleII);
\draw[->,thick] (middleI) -- (bottom); 
\draw[->, thick] (middleII) -- (bottom);
\node (label) at (0,0) [circle, inner sep=2pt, color=black, fill=brown!25!] {$\projPBpNAME{\gamma}{i}{\ell}$};
\end{tikzpicture} $\; .$
\hspace{3cm} $\text{}$
\end{center}

The category $\Rep^{0}(\Ap)$ is a rigid monoidal category with tensor product $\ind(U) \otimes \ind(V) \cong \ind(U \otimes V).$ $\Rep^{0}(\Ap)$ is also braided with braiding $c^{\Ap}_{\ind(U),\ind(V)}$ defined by
\begin{align}
c_{\ind(U),\ind(V)}^{\Ap}= \Id_{\Ap} \otimes c_{U,V},\\
\end{align}
where $c_{U,V}$ is the braiding on $\mathcal{C}_{\oplus}$ given by the product of the braidings on $\mathcal{H}$-Mod and $\UHbar{2}$-Mod,
where we have implicitly assumed the isomorphism $\ind(U)\otimes\ind(V)\cong\ind(U\otimes V)$.
 If $p$ is odd, then $\Rep^{0}(\Ap)$ has twist $\theta_{\Ap}$ and Hopf links $S_{\ind(U),\ind(V)}^{\hopflink}$ given by
\begin{align}
\theta_{\ind(V)}&=\Id_{\Ap} \otimes \theta_V\\
\sS_{\ind(U),\ind(V)}^{\hopflink}&= \sS_{U,V}^{\hopflink} 
\end{align}
where $\theta_V$ and $S_{U,V}^{\hopflink}$ are the twist and Hopf links respectively on $\mathcal{C}_{\oplus}$, and we are viewing the Hopf links as the scalars by which they act.\\
\end{theorem}

\subsubsection{The examples of the $\beta\gamma$ \voa{} and $\mathbb L_{-4/3}(\mathfrak{sl}_2)$}

The $\Bp$-algebra for $p=2$ and $p=3$ are the 
$\beta\gamma$ \voa{} and the affine \voa{} $\mathbb L_{-4/3}(\mathfrak{sl}_2)$ of $\mathfrak{sl}_2$ at level $-4/3$, respectively. These \voas{} have been discussed at length in \cite{A2, CR1, CR2, CR4, CRW, Ri, Ri2, Ri3, RW2}. We use the notation of \cite{CRW}. Since these two are our first two examples and since they are presently the most important ones in applications we now list their representation-theoretic data more explicitly. We especially list tensor product decompositions and braidings. 

\paragraph{The $\beta\gamma$ \voa}
${}$

The $\beta\gamma$ \voa{} has modules $\mathbb E_\lambda^s, \mathbb L_0^s$ with $s\in \mathbb Z$ and $\lambda \in \mathbb R/\mathbb Z$. The $\beta\gamma$ \voa{} itself is $\mathbb L_0^0$ and this is the only highest-weight module. The $\mathbb E_\lambda^0$ are the relaxed-highest weight modules, i.e., their conformal weight is bounded below but the top level is infinite-dimensional. The supersript $s$ denotes the twists of these modules by an automorphism called the spectral flow. Comparing with our notation we first set $\lambda_2 = i = \sqrt{-1}$ and then under the conjectural quantum group to singlet \voa{} equivalence one has 
\begin{align}
\mathbb E_\lambda^s  \ \leftrightarrow \ E^V_{\frac{2s-1-\lambda}{2\sqrt{-1}},\lambda}   \qquad \text{and} \qquad  \mathbb L_0^s  \ \leftrightarrow \ E^S_{\frac{s}{\sqrt{-1}}, 0, 0}. 
\end{align}
We now list the data of $\Rep^{0}(\mathcal{A}_2)$ explicitly including  tensor product decompositions and braidings. 
We have the set $\ddot{\mathbb{C}}=(\mathbb{C} \setminus \mathbb{Z}) \cup 2\mathbb{Z}$. Hence, the simple modules in $\Rep^{0}(\mathcal{A}_2)$ are
\begin{align*}
 \simTypiBpNAME{\gamma}{\alpha} &= \ind\left(\FV{\gamma}{\alpha}\right)& \text{with } \alpha \in &\CCdot,\text{ and }  \frac{\alpha+1}{2}+i\gamma \in \ZZ \; , \\
E_{\gamma,0,\ell}^S&=\ind \left( \mathrm{F}_{\gamma} \boxtimes \mathbb{C}_{2\ell }^H \right)&   \text{with } i\gamma,\ell &\in \mathbb{Z} \; .
\end{align*}
We have families of indecomposable modules:
\begin{align*}
\projVBpNAME{\gamma}{\alpha} &= \ind\left(\FV{\gamma}{\alpha}\right) &\text{with } \alpha \in &\mathbb{Z}\setminus2\mathbb{Z} \text{ and } \frac{\alpha+1}{2}+i\gamma \in \ZZ \; , \\
Q_{\gamma,0,\ell}^P &= \ind\left(\FPC{\gamma}{0}{2\ell }\right) \ &\text{with } i\gamma,\ell &\in \ZZ,
\end{align*}
where $Q_{\gamma,0,\ell}^P$ is projective, and the above modules satisfy $E^V_{\gamma,\alpha} \cong E^V_{\gamma+ik,\alpha+2k},$ $E^S_{\gamma,\ell} \cong E^S_{\gamma+ik,0,\ell+k},$ $Q^V_{\gamma,\alpha} \cong Q^V_{\gamma+ik,\alpha+2k}$, and $Q^P_{\gamma,0,\ell}  \cong Q^P_{\gamma+ik,0,\ell+k}$ for all $k \in \mathbb{Z}$. The indecomposable modules satisfy the short exact sequences:
\begin{align*}
0 \to E_{\gamma,0,\ell}^S \to Q^V_{\gamma,1+2\ell} \to E_{\gamma,0,\ell+1}^S \to 0\qquad\text{and}\qquad
0 \to E_{\gamma,1+2\ell}^V \to Q_{\gamma,0,\ell}^P \to E_{\gamma,1+2(\ell-1)}^V \to 0.
\end{align*}
The tensor decompositions for $\mathcal{C}$ can be found in \cite[Section 8]{CGP} and it then follows that 
\begin{align*}
E_{\gamma_1,0,\ell_1}^S \otimes E_{\gamma_2,0,\ell_2}^S \cong E_{\gamma_1+\gamma_2,0,\ell_1+\ell_2}^S\qquad\text{and}\qquad
E_{\gamma_1,0,\ell}^S \otimes E_{\gamma_2,\alpha}^V \cong E_{\gamma_1+\gamma_2,\alpha+2\ell}^V
\end{align*}
with $\alpha \in \ddot{\mathbb{C}}$. For $\alpha,\beta \in \ddot{\mathbb{C}}$ with $\alpha+\beta \not \in \mathbb{Z}$, we have
\[E_{\gamma_1,\alpha}^V \otimes E_{\gamma_2,\beta}^V  \cong  E_{\gamma_1+\gamma_2,\alpha+\beta-1}^V \oplus E_{\gamma_1+\gamma_2, \alpha+\beta+1}^V.\]
When $\alpha+\beta=n \in \mathbb{Z}$, set $n=j+2k$ with $j=0,1$ and $k\in \mathbb{Z}$. The tensor decompositions for these case do not appear in \cite{CGP}, but are easily computed from characters for $p=2,3$. If $\alpha \in \ddot {\mathbb{C}}$, $n=2\ell$, we have
\[ E_{\gamma_1,\alpha}^V \otimes E_{\gamma_2,\beta}^V \cong Q_{\gamma_1+\gamma_2,0,\ell}^P\]
and if $n=1+2\ell$, we have
\[ E_{\gamma_1,\alpha}^V \otimes E_{\gamma_2,\beta}^V \cong E_{\gamma_1+\gamma_2,2\ell}^V \oplus E_{\gamma_1+\gamma_2,2(\ell+1)}^V \]

Let $c_{X,Y}$ denote the braiding. We have 
\[ c_{\mathrm{F}_{\gamma_1} \boxtimes X,\mathrm{F}_{\gamma_2} \boxtimes Y}=c_{\mathrm{F}_{\gamma_1},\mathrm{F}_{\gamma_2}} \boxtimes c_{X,Y}=e^{\pi i \gamma_1 \gamma_2}\mathrm{Id} \boxtimes c_{X,Y}.\]
The braiding restricted to a simple summand is a scalar which
 can be computed by acting with the braiding on a highest (or lowest) weight vector in the summand and we get
\begin{align*}
c_{E_{\gamma_1,0,\ell_1}^S,E_{\gamma_2,0,\ell_2}^S}=q^{2(\ell_1\ell_2+\gamma_1\gamma_2)}\mathrm{Id}_{E_{\gamma_1+\gamma_2,0,\ell_1+\ell_2}^S}\qquad\text{and}\qquad
c_{E_{\gamma_1,0,\ell}^S,E_{\gamma_2,\alpha}^V}=q^{\ell(\alpha+1)+2\gamma_1\gamma_2}\mathrm{Id}_{E_{\gamma_1+\gamma_2,\alpha+2\ell}^V}.
\end{align*}
For $\alpha, \beta \in \ddot{\mathbb{C}}$ with $\alpha+\beta \not \in \mathbb{Z}$, we have
\[
c_{E_{\gamma_1,\alpha}^V,E_{\gamma_2,\beta}^V} =q^{\frac{1}{2}(\alpha+1)(\beta+1)+2\gamma_1\gamma_2}\mathrm{Id}_{E_{\gamma_1+\gamma_2,\alpha+\beta+1}^V} \oplus q^{\frac{1}{2}(\alpha-1)(\beta-1)+2\gamma_1\gamma_2}\mathrm{Id}_{E_{\gamma_1+\gamma_2,\alpha+\beta-1}^V}.\]
If $\alpha+\beta=2\ell$, we have
\[ c_{E_{\gamma_1,\alpha}^V \otimes E_{\gamma_2,\beta}^V}=q^{\frac{1}{2}(\alpha+1)(\beta+1)+2\gamma_1\gamma_2}\mathrm{Id}_{Q_{\gamma_1+\gamma_2,0,\ell}^P}  \oplus n_{Q_{\gamma_1+\gamma_2,0,\ell}^P} \]
for some nilpotent endomorphism $n_{Q_{\gamma_1+\gamma_2,0,\ell}^P}$ on $Q_{\gamma_1+\gamma_2,0,\ell}^P$.
If $\alpha+\beta=1+2\ell$, we have
\[ c_{E_{\gamma_1,\alpha}^V,E_{\gamma_2,\beta}^V}=q^{\frac{1}{2}(\alpha+1)(\beta+1)+2\gamma_1\gamma_2}\mathrm{Id}_{E_{\gamma_1+\gamma_2,2(\ell+1)}^V} \oplus q^{\frac{1}{2}(\alpha-1)(\beta-1)+2\gamma_1\gamma_2}\mathrm{Id}_{E_{\gamma_1+\gamma_2,2\ell}^V}\]

\paragraph{The affine \voa{} $\mathbb L_{-4/3}(\mathfrak{sl}_2)$ of $\mathfrak{sl}_2$ at level $-4/3$} ${}$

$\mathbb L_{-4/3}(\mathfrak{sl}_2)$ has  modules $\mathbb E_\lambda^s, \mathbb L_0^s, , \mathbb L_{-2/3}^s$ with $s\in \mathbb Z$ and $\lambda \in \mathbb R/2\mathbb Z$. As before $s$ indicates the spectral flow twists of modules. The $\mathbb E_\lambda^0$ are the relaxed-highest weight modules and the affine \voa{} itself is $\mathbb L_0^0$. There are two more highest and lowest weight modules, namely $\mathbb L_{-2/3}^0$  and $\mathbb L_0^1$ are of highest-weight $-2/3$ and $-4/3$ while $\mathbb L_{-2/3}^{-1}$  and $\mathbb L_0^{-1}$ are of lowest-weight $2/3$ and $4/3$ (here the highest and lowest weights refer to the $\mathfrak{sl}_2$-weights.). The identification with our modules is with $\lambda_3 = \sqrt{-3/2}$ given by
\begin{align}
\mathbb E_\lambda^s  \ \leftrightarrow \ E^V_{\frac{2s-\lambda}{2\sqrt{-3/2}},\lambda}  \qquad \text{and} \qquad  \mathbb L_0^s  \ \leftrightarrow \ E^S_{\frac{s}{\sqrt{-3/2}}, 0, 0} \qquad \text{and} \qquad  \mathbb L_{-2/3}^s  \ \leftrightarrow \ E^S_{\frac{2s+1}{\sqrt{-3/2}}, 1, 0}. 
\end{align}

We have $\ddot{\mathbb{C}}=(\mathbb{C}\setminus \mathbb{Z}) \cup 3\mathbb{Z}$. The simple modules in $\Rep^{0}(\mathcal{A}_3)$ are
\begin{align*}
 \simTypiBpNAME{\gamma}{\alpha} &= \ind\left(\FV{\gamma}{\alpha}\right)& \text{with } \alpha &\in \CCdot \text{ and }  \alpha+2+\sqrt{6}i\gamma \in 2\ZZ \; , \\
 \simAtypiBpNAME{\gamma}{j}{\ell} &= \ind\left(\FSC{\gamma}{j}{\ell p}\right)&  \text{with } j &\in \{0,1\}, \ell \in \mathbb{Z} \text{ and } j+3\ell+\sqrt{6}i\gamma  \in 2\ZZ \; .
\end{align*}
We have families of indecomposable modules:
\begin{align*}
\projVBpNAME{\gamma}{\alpha} &= \ind\left(\FV{\gamma}{\alpha}\right) &\text{with } \alpha &\in \mathbb{Z} \setminus 3\mathbb{Z} \text{ and } \alpha+2+\sqrt{6}i\gamma \in 2\ZZ \; , \\
\projPBpNAME{\gamma}{j}{\ell} &= \ind\left(\FPC{\gamma}{j}{\ell p}\right) \ &\text{with } j &\in \{0,1\},\ell \in \mathbb{Z} \text{ and } \; j+3\ell+\sqrt{6}i\gamma  \in 2\ZZ,
\end{align*}
with $\projPBpNAME{\gamma}{j}{\ell}$ being projective, and the above modules satisfy $E^V_{\gamma,\alpha} \cong E^V_{\gamma+i\sqrt{3/2}k,\alpha+3k},$ $E^S_{\gamma,j,\ell} \cong E^S_{\gamma+i\sqrt{3/2}k,j,\ell+k},$ $Q^V_{\gamma,\alpha} \cong Q^V_{\gamma+i\sqrt{3/2}k,\alpha+3k}$, and $Q^P_{\gamma,j,\ell}  \cong Q^P_{\gamma+i\sqrt{3/2}k,j,\ell+k}$ for all $k \in \mathbb{Z}$. The reducible indecomposables satisfy the following short exact sequences:
\begin{align*}
 0 \to E_{\gamma,2-j, \ell}^S \to Q_{\gamma,j+3\ell}^V \to E_{\gamma, j-1,\ell+1} \to 0\qquad\text{and}\qquad
0 \to E_{\gamma,2-j+3\ell}^V \to Q_{\gamma,j,\ell}^P \to E_{\gamma,1+j+3(\ell-1)} \to 0.
\end{align*}
The tensor decompositions are as follows:
\begin{align*}
E_{\gamma_1,i,\ell_1}^S \otimes E_{\gamma_2,j,\ell_2}^S \cong \bigoplus\limits_{\substack{ k=|i-j| \\ \text{by $2$}}}^{i+j} E_{\gamma_1+\gamma_2,k,\ell_1+\ell_2}^S\qquad\text{and}\qquad
E_{\gamma_1,0,\ell}^S \otimes E_{\gamma_2,\alpha}^V \cong E_{\gamma_1+\gamma_2,\alpha+3\ell}^V.
\end{align*}
Note that when $i=j=1$ above, we have $E_{\gamma_1+\gamma_2,2,\ell_1+\ell_2}^S \cong E_{\gamma_1+\gamma_2,3(\ell_1+\ell_2)}^V$. When $\alpha \in \mathbb{C}\setminus \mathbb{Z}$ or $\alpha=3\ell_2$, we get 
\begin{align*} E_{\gamma_1,1,\ell}^S \otimes E_{\gamma_2,\alpha}^V \cong E_{\gamma_1+\gamma_2,\alpha+1+3\ell}^V \oplus E_{\gamma_1+\gamma_2,\alpha-1+3\ell}^V\qquad\text{and}\qquad
E_{\gamma_1,1,\ell_1}^S \otimes E_{\gamma_2,3\ell_2}^V \cong Q_{\gamma_1+\gamma_2,1,\ell_1+\ell_2}^P
\end{align*}
respectively. For $\alpha,\beta \in \ddot{\mathbb{C}}$ with $\alpha + \beta \not \in \mathbb{Z}$, we have
\[ E_{\gamma_1,\alpha}^V \otimes E_{\gamma_2,\beta}^V \cong \bigoplus\limits_{\substack{ k=-2 \\ \text{by $2$}}}^{2} E_{\gamma_1+\gamma_2,\alpha+\beta+k}^V. \]
For $\alpha \in \ddot{\mathbb{C}}$, $\alpha+\beta=n$ with $n=3\ell, 1+3\ell$, and $2+3\ell$, we get
\begin{align*}
E_{\gamma_1,\alpha}^V \otimes E_{\gamma_2,\beta}^V &\cong Q_{\gamma_1+\gamma_2,0,\ell}^P \oplus E_{\gamma_1+\gamma_2,3\ell}^V\\
E_{\gamma_1,\alpha}^V \otimes E_{\gamma_2,\beta}^V &\cong Q_{\gamma_1+\gamma_2,1,\ell}^P \oplus E_{\gamma_1+\gamma_2,3(\ell+1)}^V\\
E_{\gamma_1,\alpha}^V \otimes E_{\gamma_2,\beta}^V &\cong Q_{\gamma_1+\gamma_2,1,\ell+1}^P \oplus E_{\gamma_1+\gamma_2,3\ell}^V\\
\end{align*}
respectively. We immediately obtain the following braidings:
\begin{align*}
c_{E_{\gamma_1,0,\ell_1}^S,E^S_{\gamma_2,i,\ell_2}}&=q^{\frac{1}{2}3\ell_1(3\ell_2+i)+3\gamma_1\gamma_2}\mathrm{Id}_{E_{\gamma_1+\gamma_2,i,\ell_1+\ell_2}^S}\\
c_{E_{\gamma_1,1,\ell_1}^S,E_{\gamma_2,1,\ell_2}^S}&=q^{\frac{1}{2}(3\ell_1+1)(3\ell_2+1)+3\gamma_1\gamma_2}\mathrm{Id}_{E_{\gamma_1+\gamma_2,3(\ell_1+\ell_2)}^V} \oplus q^{\frac{1}{2}(3\ell_1-1)(3\ell_2-1)+3\gamma_1\gamma_2}\mathrm{Id}_{E_{\gamma_1+\gamma_2,0,\ell_1+\ell_2}^S}\\
c_{E_{\gamma_1,0,\ell}^S,E_{\gamma_2,\alpha}^V}&= q^{\frac{3\ell}{2}(\alpha+2)+3\gamma_1\gamma_2}\mathrm{Id}_{E_{\gamma_1+\gamma_2,\alpha+3\ell}^V}\\
c_{E_{\gamma_1,1,\ell}^S \otimes E_{\gamma_2,\alpha}^V} &=q^{\frac{1}{2}(1+3\ell)(\alpha+2)+3\gamma_1\gamma_2}\mathrm{Id}_{E_{\gamma_1+\gamma_2,\alpha+1+3\ell}^V} \oplus q^{\frac{1}{2}(-1+3\ell)(\alpha-2)+3\gamma_1\gamma_2}\mathrm{Id}_{E_{\gamma_1+\gamma_2,\alpha-1+3\ell}^V}\\
c_{E_{\gamma_1,1,\ell_1}^S \otimes E_{\gamma_2,3\ell_2}^V} & =q^{\frac{1}{2}(1+3\ell_1)(2+3\ell_2)+3\gamma_1\gamma_2}\mathrm{Id}_{Q_{\gamma_1+\gamma_2,1,\ell_1+\ell_2}^P}\oplus n_{Q_{\gamma_1+\gamma_2,1,\ell_1+\ell_2}^P}\\
\end{align*}
with $n_{Q_{\gamma_1+\gamma_2,1,\ell_1+\ell_2}^P}$ a nilpotent endomorphism on $Q_{\gamma_1+\gamma_2,1,\ell_1+\ell_2}^P$.
 For $\alpha,\beta \in \ddot{\mathbb{C}}$ with $\alpha+\beta \not \in \mathbb{Z}$, we have
\begin{align*}
c_{E_{\gamma_1,\alpha}^V,E_{\gamma_2,\beta}^V}&=q^{\frac{1}{2}(\alpha+2)(\beta+2)+3\gamma_2\gamma_2}\mathrm{Id}_{E_{\gamma_1+\gamma_2,\alpha+\beta+2}^V} \oplus q^{\frac{1}{2}(\alpha-2)(\beta-2)+3\gamma_2\gamma_2}\mathrm{Id}_{E_{\gamma_1+\gamma_2,\alpha+\beta-2}^V}\\
& \oplus \left(q^{\frac{1}{2}(\alpha+2)(\beta-2)}+q^{\frac{1}{2}\alpha\beta}+q^{\frac{1}{2}(\alpha-2)(\beta+2)}\right)q^{3\gamma_1\gamma_2}\mathrm{Id}_{E_{\gamma_1+\gamma_2,\alpha+\beta}^V}
\end{align*}
For $\alpha \in \ddot{\mathbb{C}}$, $\alpha+\beta=n$ with $n=3\ell,1+3\ell$, and $2+3\ell$, we get
\begin{align*}
c_{E_{\gamma_1.\alpha}^V \otimes E_{\gamma_2,\beta}^V} &=\left( q^{\frac{1}{2}(\alpha+2)\beta}+q^{\frac{1}{2}\alpha(\beta+2)}\right)q^{3\gamma_1\gamma_2}\mathrm{Id}_{Q_{\gamma_1+\gamma_2,0,\ell}^P} \oplus n_{Q_{\gamma_1+\gamma_2,0,\ell}^P} \oplus q^{\frac{1}{2}(\alpha+2)(\beta+2)+3\gamma_1\gamma_2}\mathrm{Id}_{E_{\gamma_1+\gamma_2,3\ell}^V} \\
c_{E_{\gamma_1.\alpha}^V \otimes E_{\gamma_2,\beta}^V}&=\left(q^{\frac{1}{2}(\alpha+2)\beta}+q^{\frac{1}{2}\alpha(\beta+2)}\right)q^{3\gamma_1\gamma_2}\mathrm{Id}_{Q_{\gamma_1+\gamma_2,1,\ell}^P} \oplus n_{Q_{\gamma_1+\gamma_2,1,\ell}^P} \oplus q^{\frac{1}{2}(\alpha+2)(\beta+2)+3\gamma_1\gamma_2}\mathrm{Id}_{E_{\gamma_1+\gamma_2,3(\ell+1)}^V}\\
c_{E_{\gamma_1.\alpha}^V \otimes E_{\gamma_2,\beta}^V}&=\left( q^{\frac{1}{2}(\alpha+2)(\beta-2)}+q^{\frac{1}{2}(\alpha\beta}+ q^{\frac{1}{2}(\alpha-2)(\beta+2)}\right)q^{3\gamma_1\gamma_2}\mathrm{Id}_{E_{\gamma_1+\gamma_2,3\ell}^V}\\
& \qquad \oplus q^{\frac{1}{2}(\alpha+2)(\beta+2)+3\gamma_1\gamma_2}\mathrm{Id}_{Q_{\gamma_1+\gamma_1,1,\ell+1}^P}  \oplus n_{Q_{\gamma_1+\gamma_1,1,\ell+1}^P}
\end{align*}
for some nilpotent endomorphisms on the projective modules $Q^P$.

\subsubsection{Modularity and Verlinde's formula}

We divide modules into typical modules $\simTypiBpNAME{\gamma}{\alpha}$ and atypical modules  $\simAtypiBpNAME{\gamma}{i}{\ell}$. Modules of the $\Bp$-algebra are bigraded by conformal weight and also by the weight of the Heisenberg \voa. The corresponding graded trace turns out to only make sense as a formal power series in the case of typical module characters as formal delta distributions appear. Using the ideas of \cite{CR2, CR3} we can nonetheless compute a modular $S$-transformation giving us a certain function on the set of typical modules that we call $S$-kernel $\sS^\chi$.

 Atypical characters on the other hand turn out to converge in certain domains and then can be meromorphcally continued to components of vector-valued meromorphic Jacobi forms. It turns out, however, that different modules have the same (up to a possible sign) meromorphic continuation. This meromorphic vector-valued Jacobi form can be related to the semi-simplification of $\Rep^{0}(\Ap)$. For this recall that the semi-simplification of a category is the category obtained by quotienting negligible morphisms. To avoid confusion let us list the relevant definitions 
\begin{definition} Let $\mathcal C$ be a rigid braided tensor category. Let $M, N$ be objects, then a morphism $f:M \rightarrow N$ is negligible if for every morphism $g:N\rightarrow M$ the trace of $f\circ g$ vanishes. The semisimplification $\mathcal C^{\text{ss}}$ of $\mathcal C$ is the category whose objects are those of $\mathcal C$ but all negligible morphism are identified with the zero morphism. An object $M$ is called negligible if the identity on $M$ is a negligible morphism and in a rigid tensor category the subcategory $\mathcal N$ whose objects are all negligible objects forms a tensor ideal. 

Let $\mathcal G(\mathcal C)$ be the Grothendieck ring of $\mathcal C$ and let $\mathcal G(\mathcal N)$ be the Grothendieck ring of the ideal of negligible objects then we define the ring
\[
\mathcal G^{\text{ss}}(\mathcal C) := \mathcal G(\mathcal C)/\mathcal G(\mathcal N)
\]
and we note that in general $\mathcal G^{\text{ss}}(\mathcal C)$ is a homomorphic image of $\mathcal G(\mathcal C^{\text{ss}})$.
\end{definition}

The modular properties of characters of the $\Bp$-algebra and Hopf links of $\C$ are studied and compared in Sections \ref{Sectiontypical} and \ref{Sectionatypical}. In Section \ref{Sectiontypical}, the modular $S$-matrix $\sS^{\chi}$ coming from the modular action on characters for typical modules is computed and shown to agree with the $S$-matrix $\sS^{\hopflink}$ coming from closed Hopf-links associated to typical modules in $\C$ up to normalised conjugation in Proposition \ref{typicalcomparison}. To compare the atypical modules, the ring $\mathcal{G}^{\text{ss}}(\C^{0})$ of the category of local modules (those which are induced to $\mathrm{Rep}^0\Ap$ by the induction functor) is derived in Proposition \ref{Grothendieck}. The corresponding matrix $\sS^{\hopflink}$ is derived and shown to agree up to normalised conjugation with the matrix $\sS^{\chi}$ determined by the modular action on the Verlinde algebra of characters generated by atypical $\Bp$-modules when $p$ is odd (see Proposition \ref{atypicalhopf}). From this, the Verlinde formula immediately follows in Corollary \ref{Verlindeodd}. When $p$ is even, $\Bp$ is half-integer graded and we instead compare modular properties associated to its integer part $\Bp^{\overline{0}}$, showing again that the statement analagous to the following Theorem \ref{introcomparison} holds. 
\begin{theorem} \textup{(Verlinde's formula)}\label{thm:Verintro}

Let $\ast$ denote complex conjugation of the entries.
For the parametrization of atypical simples,  refer to the discussion around \eqref{Lambdap}.
 \label{introcomparison} 
 \begin{enumerate}
\item  \textup{(Proposition \ref{typicalcomparison})} Normalized character $\sS^{\chi}$ and Hopf link $\sS^{\hopflink}$ of typical modules agree up to complex conjugation, 
\[ 
\frac{\sS^{\chi \ast}_{(\nu, \ell),(\nu ', \ell ')}}{\sS^{\chi \ast}_{\mathds{1},(\nu ', \ell ')}}= \frac{\sS^{\hopflink}_{(\nu,\ell),(\nu ', \ell ')}}{\sS^{\hopflink}_{\mathds{1},(\nu ', \ell ')}}.\]
\item \textup{(Proposition \ref{atypicalhopf})} Let $p$ be odd.
Normalized Hopf link $\mathsf{S}^{\hopflink}$ and character $\mathsf{S}^{\chi}$ of atypical modules are in agreement up to complex conjugation,
\[
\frac{\sS^{\chi \ast}_{(s,s'),(n,n')}}{\sS^{\chi \ast}_{\mathds{1},(n,n')}}= \frac{\sS^{\hopflink}_{(s,s'),(n,n')}}{\sS^{\hopflink}_{\mathds{1},(n,n')}}.
 \]
\item \textup{(Corollary \ref{Verlindeodd})} Let $p$ be odd and let $\Lambda_p$ be a $\mathbb Z$-basis of $\mathcal G^{\text{ss}}(\mathrm{Rep}^0\Ap)$ with structure constants $N^{(k,k')}_{(s,s'),(t,t')}$ that is
\[
(k, k') \times (s, s') = \sum_{(n,n') \in \Lambda_p}N^{(k,k')}_{(s,s'),(t,t')} (t, t')
\]
for $(k, k')$ and $(s, s')$ in $\Lambda_p$.
Then the Verlinde formula holds 
\[
\sum\limits_{(n,n') \in \Lambda_p} \frac{\sS^{\hopflink}_{(s,s'),(n,n')}\sS^{\hopflink}_{(t,t'),(n,n')}(\sS^{\hopflink})^{-1}_{(n,n'),(k,k')}}{\sS^{\hopflink}_{\mathds{1},(n,n')}}=N^{(k,k')}_{(s,s'),(t,t')},
\]
\end{enumerate}
\end{theorem}
We remark that the complex conjugation indicates that our conjectural relation between representation categories of unrolled quantum group and singlet algebra is a braid-reversed equivalence, see Remark \ref{remark:cc}.

It was conjectured in \cite[Remark 5.6]{C} that the $\Bp$ \voa{} is a quantum Hamiltonian reduction of $V_k(\mathfrak{sl}_{p-1})$ at level $k+p-1=\frac{p-1}{p}$.
Quantum Hamiltonian reductions are associated to nilpotent elements and the relevant nilpotent element $f$ for us corresponds to the partion $(p-2, 1)$ of $p-1$. We denote the corresponding simple $W$-algebra by $W_k(\mathfrak{sl}_{p-1}, f)$, 
 In Section \ref{SectionQH}, we confirm this conjecture up to character. That is, we show that (Theorem \ref{BpQH}):

\begin{theorem} 
The characters of the $\Bp$-algebra and of $W_k(\mathfrak{sl}_{p-1}, f)$ coincide for $k=-p+1+\frac{p-1}{p}$.
\end{theorem}

\subsection*{Acknowledgements}
We thank Andrew Ballin for his comments.
 TC is supported by NSERC $\#$RES0020460. SK is supported by a start-up grant provided by University of Denver.

\section{Background}

We will use simple currents contained in the Deligne product of the representation categories of the Singlet and Heisenberg \voas{} to construct a \voa{} known as the $\Bp$-algebra (see \cite{C},\cite{CRW}). It was shown in \cite{CMR} that there exists a bijection between a particular category of modules for a \voa{} called the singlet algebra and the category of weight modules for the unrolled restricted quantum group of $\sln{2}$, $\UHbar{2}$.
Further structures as tensor products and open Hopf links were succesfully matched with conjectural fusion producs on the singlet algebra and asymptotic dimensions of characters, which led to the conjecture that these categories are equivalent as ribbon categories. Further evidence for this conjecture has been given in \cite{CGR}.

 In this section we recall the definition of $\UHbar{2}$ and the structure of its category of weight modules, as seen in \cite{CGP}. We will also recall the definition of the Heisenberg \voa{} and the category of modules we are interested in.

\subsection{The unrolled restricted quantum group of $\sln{2}$ and its weight modules}

Let $p \geq 2$ be a positive integer and 
\begin{align}
q=e^{\pi i/p}
\end{align}
a $2p$-th root of unity. For any $x \in \mathbb{C}$ we choose the notation
\begin{align}
 \{x\}=q^x-q^{-x}, [x]=\frac{\{x\}}{\{1\}}, \text{ and for any $n \in \mathbb{Z}$, } \{n\}!=\{n\}\{n-1\}...\{1\}. 
 \end{align}
The quantum group associated to $\sln{2}$, $\U{2}$ is the associative algebra over $\CC$ with generators $E,F,K,K^{-1}$ and relations
\[
KK^{-1}=K^{-1}K=1, \quad KE=q^2EK, \quad KF=q^{-2}FK, \quad  [E,F]=\frac{K-K^{-1}}{q-q^{-1}}.
\]

This algebra has Hopf algebra structure given by counit $\epsilon: \U{2} \to \CC$, coproduct $\Delta: \U{2} \to \U{2} \otimes \U{2}$, and antipode $S:\U{2} \to \U{2}$ defined by
\begin{align*}
\Delta(K)&=K \otimes K, & \epsilon(K)&=1, & S(K)&=K^{-1},\\
\Delta(E)&=1 \otimes E + E \otimes K, & \epsilon(E)&=0, & S(E)&=-EK^{-1},\\
\Delta(F)&=K^{-1} \otimes F+F \otimes 1, & \epsilon(F)&=0, & S(F)&=-KF.
\end{align*}

The unrolled quantum group of $\sln{2}$, $\UH{2}$, is defined by extending $\U{2}$ through the addition of a fifth generator $H$ with relations
\[ HK^{\pm 1}=K^{\pm 1}H, \qquad [H,E]=2E, \qquad [H,F]=-2F. \]
The counit, coproduct, and antipode can be extended to $U_q^H(\mathfrak{sl}_2)$ by defining
\[ \Delta(H)=H \otimes 1 + 1\otimes H, \qquad \epsilon(H)=0, \qquad S(H)=-H. \]

The unrolled restricted quantum group of $\sln{2}$, $\UHbar{2}$, is then obtained taking the quotient of $\UH{2}$ by the relations $E^p=F^p=0$.\\
A finite dimensional $\UHbar{2}$-module V is called a weight module if it is a direct sum of its $H$-eigenspaces ($H$ acts semisimply) and $K=q^H$ as an operator on V. Let $\UHbar{2}$-Mod denote the category of weight modules for $\UHbar{2}$. A classification of simple and projective modules was given in \cite{CGP} as follows:\\

Given any $n \in \{0,...,p-1\}$, let $S_n$ be the simple highest weight module of weight $n$ and dimension $n+1$ with basis $\{s_0,...,s_n\}$ and action
\[Fs_i=s_{i+1}, \quad Es_i=[i][n+1-i]s_{i-1}, \quad Hs_i=(n-2i)s_i, \quad Es_0=Fs_n=0. \]

For any $\alpha \in \CC$, define $V_{\alpha}$ to be the $p$-dimensional highest weight module of highest weight $\alpha+p-1$, whose action is defined on a basis $\{v_0,...,v_{p-1}\}$ as
\[ Fv_i=v_{i+1}, \quad Ev_i = [i][i-\alpha]v_{i-1}, \quad Hv_i=(\alpha+p-1-2i)v_i, \quad Ev_0=Fv_{p-1}=0.\\ \]

$V_{\alpha}$ is called typical if $\alpha \in \ddot{\CC}:=(\CC-\ZZ) \cup p\ZZ$ and atypical otherwise. The typical $V_{\alpha}$ are simple since any basis vector $v_i$ can generate a scalar multiple of every other basis vector through the action of $E$ and $F$. If $V_{\alpha}$ is atypical, then we have $\alpha=pm+k$ for some $m \in \ZZ$ and $1 \leq k \leq p-1$. Hence,
\[E v_{k}=-[k][k-(pm+k)]v_{k-1}=[k][pm]v_{k-1}=0,\]
since $\{pm\}=q^{pm}-q^{-pm}=(-1)^m-(-1)^{-m}=0$. So, when $V_{\alpha}$ is atypical, it contains a simple submodule generated by the basis elements $\{v_k,v_{k+1},...,v_{p-1}\}$. \\

For any $\ell \in \ZZ$, let $\CC^H_{\ell p}$ denote the one dimensional module on which $E$ and $F$ act as zero and $H$ acts as scalar multiplication by $\ell p$. Then the following holds

\begin{proposition}$\textup{\cite[Theorem 5.2 and Lemma 5.3]{CGP}}$
\begin{enumerate}
\item  The typical $V_{\alpha}$ are projective.
\item Every simple module in $\UHbar{2}$-Mod is isomorphic to $\SC{n}{\ell p}$ for some $n \in \{0,...,p-2\}$ and $\ell \in \ZZ$ or $V_{\alpha}$ for some $\alpha \in \CCdot$.
\end{enumerate}
\end{proposition}

A weight vector $v$ is called dominant if $(FE)^2v=0$. If $v$ is a dominant weight vector of weight $i$, then we denote by $P_i$ the module generated by $v$. This module's structure is given explicitly in \cite[Section 6]{CGP}, and the following proposition was proven therein:

\begin{proposition}
The module $P_i$ is projective and indecomposable with dimension $2p$. Any projective indecomposable module with integer highest weight $(\ell+1)p-i-2$ is isomorphic to $\PC{i}{\ell p}$.
\end{proposition}

If $V$ is an object in $\UHbar{2}$-Mod with basis $\{v_1,...,v_n\}$, then $V$ has the obvious dual vector space $V^*=\Hom_{\CC}(V,\CC)$ with basis $\{v_1^*,...,v_n^*\}$ and action $af(v)=f(S(a)v)$ for $f \in V^*, a \in \UHbar{2}$. The left duality morphisms are given by
\[\overrightarrow{\text{coev}}_V:\CC \rightarrow V \otimes V^* \quad \text{ and } \quad \overrightarrow{\text{ev}}_V:V^* \otimes V \rightarrow \CC,\]
where $ \overrightarrow{\text{coev}}(1)=\sum\limits_i^n v_i \otimes v_i^*$, and $\overrightarrow{\text{ev}}(f \otimes w)=f(w)$. Ohtsuki defined in \cite{O} the $R$-matrix operator on $\UHbar{2}$-Mod by
\begin{align}\label{Uqbraiding}
R=q^{H \otimes H/2} \sum\limits_{n=0}^{p-1}\frac{\{1\}^{2n}}{\{n\}!}q^{n(n-1)/2}E^n \otimes F^n,
\end{align}

where $q^{H \otimes H/2}(v \otimes w)=q^{\lambda_v \lambda_w/2}v \otimes w$ for weight vectors $v,w$ with weights $\lambda_v$ and $\lambda_w$. The braiding on $\UHbar{2}$-Mod is then given by the family of maps $c_{V,W}:V \otimes W \rightarrow W \otimes V$ where $c_{V,W}(v \otimes w) = \tau(R(v \otimes w))$ where $\tau$ is the flip map $w \otimes v \mapsto v \otimes w$. Ohtsuki also defined an operator $\widetilde{\theta}_V:V \rightarrow V$ on each $V \in \UHbar{2}\text{-Mod}$ by
\begin{align}\label{sl2twist}
\widetilde{\theta}=K^{p-1} \sum\limits_{n=0}^{p-1} \frac{\{1\}^{2n}}{\{n\}!}q^{n(n-1)/2} S(F^n)q^{-H^2 /2}E^n.
\end{align}

The twist $\theta_V:V \rightarrow V$ is then given by the operator $v \mapsto \widetilde{\theta}^{-1}v$. $\UHbar{2}$-Mod also admits compatible right duality morphisms
\begin{align*}
\overleftarrow{\text{ev}}_V&:V \otimes V^*: \rightarrow \CC, \qquad \overleftarrow{\text{ev}}_V(v \otimes f)=f(K^{1-p}v),\\
\overleftarrow{\text{coev}}_V&:\CC \rightarrow V^* \otimes V, \qquad \overleftarrow{\text{coev}}(1)=\sum\limits_i K^{p-1}V_i \otimes v_i^*.
\end{align*}

The following Lemma was proved in \cite[Proposition 6]{R} using \cite{CGP} and will be used in multiple results:

\begin{lemma} \label{prop:SESresolutionofAtypicals}
For any $k \in \{1,...,p-1 \}$ there are short exact sequences of modules
\begin{align*}
0 \rightarrow \SC{p-1-k}{\ell p} \rightarrow&V_{k+\ell p} \rightarrow \SC{k-1}{(\ell+1)p} \rightarrow 0,\\
0 \rightarrow V_{p-1-i+\ell p} \rightarrow P_i & \otimes \mathbb{C}_{\ell p}^H \rightarrow V_{1+i-p+\ell p} \rightarrow 0.
\end{align*}
\end{lemma}

\subsection{Heisenberg \voa{} and its Fock modules}\label{Heisenberg}

The rank-$1$ Heisenberg Lie algebra, denoted by $\Hlie$, has vector space basis given by $\{\textbf{c},b_n|n\in \ZZ \}$ and bracket
\[ [\textbf{c},b_n]=0 \text{ and } [b_n,b_m]=n\delta_{n+m,0}\textbf{c}.\\ \]
Let
$\Hlie_{\pm}=\text{Span}_{\CC}\{b_n\,|\, \pm n>0\}.$ Let $\mathcal{U}(\mathfrak{g})$ denote the universal enveloping algebra
of any Lie algebra $\mathfrak{g}$.
Denote by $\FockC{\beta}$ the usual Fock space of charge $\beta\in\CC$ with vector space basis 
$\mathcal{U}(\Hlie_{-})$
and $\Hlie$-action on an arbitrary element $b \in \FockC{\beta}$ given by
\begin{align*}
\mathbf{c} \cdot b&=b,\\
b_0 \cdot b&=\beta b,\\
b_n \cdot b&= b_nb \quad \text{for all $n<0$},\\
b_n \cdot b&=n \frac{\partial}{\partial b_{-n}} b \quad \text{for all $n>0$}.
\end{align*}

\begin{definition}
The Heisenberg \voa{} $\Heis=(\FockC{0},1,Y,T,\omega)$ is given by the following data (see \cite[Chapter 2]{BF}):
\begin{itemize}
\item a $\mathbb{Z}_+$-gradation deg($b_{j_1} \dotsb b_{j_k})=-\sum_{i=1}^k j_i$,
\item a vacuum vector $|0\rangle=1$,
\item a translation operator T defined by $T(1)$=0, $[T,b_i]=-ib_{i-1}$,
\item vertex operators $Y(-,z)$ defined by
\[ Y(b_{j_1} \dotsb b_{j_k},z)=\frac{:\partial_z^{-j_1-1}b(z) \dotsb \partial_z^{-j_k-1}b(z):}{(-j_1-1)! \dotsb (-j_k-1)!},\]
where $:X(z)Y(z):$ denotes the normally ordered product. 
\item a conformal vector $\omega=b_{-1}^2$ of central charge 1.
\end{itemize}

If $b' \in \Heis$ and $b \in \FockC{\beta}$ then $\Heis$ acts on $\FockC{\beta}$ as $b'(b)=Y(b',z)b$ as an extension of the action of $\mathfrak{h}$. We can therefore consider the representation category $\Heis$-Mod generated by the Fock spaces. This category is rigid (contains duals) and has braiding and twist. The data is
\begin{enumerate}
\item
$\FockC{\beta}^* \simeq \FockC{-\beta}$, \qquad\qquad\qquad\qquad\qquad (Duals)
\item
$\FockC{\beta_1} \otimes \FockC{\beta_2} \simeq \FockC{\beta_1+\beta_2}$, \qquad\qquad\quad \ \ (Fusion/Tensor products)
\item
$c_{\FockC{\beta_1} \otimes \FockC{\beta_2}}=e^{\pi i \beta_1 \beta_2}\Id_{\FockC{\beta_1} \otimes \FockC{\beta_2}}$, \quad\ \, (Braidings)
\item $\theta_{\FockC{\beta}}=e^{\pi i \beta^2}\Id_{\FockC{\beta}}$, \qquad\qquad\qquad\ \  \, (Ribbon Twists)
\end{enumerate}
\end{definition}

Let $\mathcal{H}^{\oplus}$ denote the category of $\CC$-graded complex vector spaces with finite or countable dimension and let $H_V:V \to V$ be the degree map on $V:= \bigoplus\limits_{\nu \in \CC} V_{\nu}$ given by $H_V|_{V_{\nu}}=\nu \Id_{V_{\nu}}$. $\mathcal{H}^{\oplus}$ can be given (non-unique) ribbon structure by defining the braiding $c$ and twist $\theta$ by
\begin{align}
\label{HRbraidtwist}c_{U,V}&:=\tau_{U,V} \circ e^{\pi i H_U \otimes H_V},\\
\theta_V&:=e^{\pi i H_V^2}\Id_{V},
\end{align}

where $\tau_{U,V}$ is the usual flip map. Denote by $\mathcal{H}^{\oplus}_{i\mathbb{R}}$ the full tensor subcategory of $\mathcal{H}^{\oplus}$ with purely imaginary index. $\HR$ is braided equivalent to the full subcategory $\Heis$-Mod whose simple objects are given by Fock modules $\FockC{iy}$ with $y\in \mathbb{R}$, or, $iy \in i \mathbb{R}$ (see \cite[Subsection 2.3]{CGR} for details). The equivalence is given by identifying the usual Fock space $\FockC{i y}$ with the one dimensional vector space of degree $iy$, $\FockH{iy}:=\mathbb{C}v_{iy}$.\\

\subsection{The Singlet \voa{} $\Mp$ and its category of modules}

Let $p \in \ZZ_{>0}$ and $L:=\sqrt{2p}\ZZ$ an even lattice. The lattice VOA $V_L:=\bigoplus_{l \in L} \FockC{l}$ associated to $L$ can be constructed via the reconstruction theorem as outlined in \cite{BF}[Proposition 5.2.5]. Let $e^{\gamma}(z)$ denote the usual fields associated to lattice \voas:

\[e^{\gamma}(z):=S_{\gamma}z^{\gamma b_0}\exp\left( -\gamma \sum\limits_{n<0} \frac{b_n}{n}z^{-n}\right)\exp\left(-\gamma \sum\limits_{n \geq 0} \frac{b_n}{n}z^{-n} \right),\]

where $S_{\gamma}$ is the shift operator $\FockC{\beta} \to \FockC{\beta+\gamma}$. Define the screening operator $\widetilde{Q}=e^{-\sqrt{\frac{2}{p}}}_0$ where $e^{\gamma}(z)=\sum\limits_{n \in \ZZ}e_n^{\gamma}z^{-n-1}$ is the Fourier expansion of $e^{\gamma}(z)$. The Singlet VOA is then defined as the kernel $\Mp:=\text{Ker}_{\FockC{0}}(\widetilde{Q})$. For $r,s\in\ZZ, 1\leq s \leq p$, let $\alpha_{r,s}=-\frac{r-1}{2}\sqrt{2p}+\frac{s-1}{\sqrt{2p}}$. For $s=p$, the Fock space with $\FockC{\alpha_{r,s}}$  is simple as an $\Mp$-module, which we denote by $\FockS{\alpha_{r,s}}$.
When $s \neq p$, $F_{\alpha_{r,s}}$ is reducible and we define $M_{r,s}$ to be the socle of $F_{\alpha_{r,s}}$ which is known to be a simple $\Mp$-module.\\

It is expected that the module categories of $\Mp$ and $\UHbar{2}$ are equivalent as monoidal (or perhaps braided) categories. This is motivated by the following statement proven in \cite{CMR}:

\begin{proposition} \label{correspondence}
{For $\alpha \in \ddot{\mathbb{C}}:=(\mathbb{C} - \mathbb{Z}) \cup p\mathbb{Z}$,
	$i\in \{0,1,\dots,p-2\}$ and $k\in\ZZ$,}
 the map 
\begin{align}\label{eqn:SimplesQuantumSinglet}
 \varphi: V_{\alpha} \mapsto \FockS{\frac{\alpha+p-1}{\sqrt{2p}}}, \qquad \varphi:\SC{i}{kp} \mapsto M_{1-k,i+1}
 \end{align}
 between simple modules of $\UHbar{2}$ and the $\Mp$ singlet \voa{} is a bijection of the sets of representatives of equivalence classes of simple modules up to isomorphisms
and  induces a morphism from the Grothendieck ring
of weight modules of $\UHbar{2}$ to the conjectured Grothendieck ring of $\Mp$ (conjectured in $\textup{\cite{CM1}}$ based on the conjectural Verlinde's formula).
\end{proposition}

A precise conjecture on the connection between the module categories of $\Mp$ and $\UHbar{2}$ is given in \cite[Subsection 3.1]{CMR} and \cite[Conjecture 5.8]{CGR}.

\subsection{Simple currents and algebra objects}\label{algebracurrents}

Conformal extensions of vertex operator algebras can be studied efficiently via the notion of (super)commutative algebra objects in vertex tensor categories. The representation category for the extended vertex operator (super)algebra then corresponds to the category of local modules for the corresponding (super)commutative algebra object. This program has been developed in \cite{KO, HKL, CKM}.
It works particularly well for simple current extensions \cite{CKL, CKLR} and our present problem falls precisely  into this framework.

\begin{definition}
A simple current is a simple object which is invertible with respect to the tensor product. Objects which are their own inverse are called self-dual.
\end{definition}

\begin{definition} \label{algebradef}
A commutative associative unital algebra (or just algebra, for short) in a braided monoidal category $\mathcal{C}$ is an object $A$ in $\mathcal{C}$ with multiplication morphism $\mu:A \otimes A \rightarrow A$ and unit $\iota: \mathds{1} \rightarrow A$ with the following assumptions:
\begin{itemize}
\item Associativity: $\mu \circ (\mu \otimes \Id_A)=\mu \circ ( \Id_A \otimes \mu) \circ a_{A,A,A}$ where $a_{A,A,A}: (A \otimes A) \otimes A \to A \otimes (A \otimes A)$ is the associativity isomorphism.
\item Unit: $\mu \circ (\iota \otimes \Id_A) \circ l_A^{-1}=\Id_A$ where $l_A:\mathds{1} \otimes A \rightarrow A$ is the left unit isomorphism.
\item Commutativity: $\mu\circ c_{A,A} = \mu$ where $c_{A,A}$ is the braiding.
\item (Optional assumption) Haploid: $\dim(\Hom_{\mathcal{C}}(\mathds{1},A))=1$.
\end{itemize}

$\rep A$ is the category whose objects are given by pairs $(V,\mu_V)$ where $V \in \text{Obj}(\mathcal{C})$ and $\mu_V \in \text{Hom}(A \otimes V,V)$ satisfying the following assumptions:
\begin{itemize}
\item $\mu_V \circ (\Id_A \otimes \mu_V)=\mu_V \circ ( \mu \otimes \Id_A)\circ a_{A,A,V}^{-1}$,
\item $\mu_V \circ (\iota \otimes \Id_V) \circ l_V^{-1}=\Id_V$.
\end{itemize}

Define $\rep^0A$ to be the full subcategory of $\rep A$ whose objects $(V,\mu_V)$ satisfy
\[ \mu_V \circ M_{A,V}=\mu_V,\]
where $M_{A,V}=c_{V,A} \circ c_{A,V}$.\\
\end{definition}

\begin{definition}\cite[Section 2]{CGR}

Let $\mathcal C$ be a tensor category with tensor identity $ \mathds{1}$ and an algebra object $A$. A morphism $\omega: A \otimes A \rightarrow  \mathds{1}$ is called a non-degenerate invariant pairing if 
\begin{enumerate}
\item The morphisms $ A \otimes (A \otimes A) \xrightarrow{\Id \otimes \mu} A \otimes A \xrightarrow{\omega}  \mathds{1}$ and $ A \otimes (A \otimes A) \xrightarrow{a_{A,A,A}^{-1}} (A \otimes A) \otimes A \xrightarrow{\mu \otimes \Id} A \otimes A \xrightarrow{\omega} \mathds{1}$ coincide. (Invariance)
\item For any object $V$ and morphism $f:V\rightarrow A$ the equalities $\omega \circ (f \otimes \Id_A) =0$ or $\omega \circ (\Id_A \otimes f)=0$ both imply that $f=0$. (non-degeneracy)
\end{enumerate}
\end{definition}

The notion of simplicity of an extended \voa{} corresponds precisely to the corresponding algebra object having a non-degenerate invariant pairing. This is explained in the proof of Corollary 5.9 of \cite{CGR}.

\begin{definition}\label{inductionfunctor}
Let $\mathcal{C}$ be a category with algebra object $A$. The induction functor $\ind: \mathcal{C} \rightarrow \Rep A$ is defined by $\ind (V)=(A \otimes V, \mu_{\ind (V)})$ where $\mu_{\ind (V)}= (\mu \otimes \Id_V)\circ a_{A,A,V}^{-1}$ (here $\mu$ is the product on $A$) and for any morphism $f$, $\ind(f)=\Id_A \otimes f$.\\
\end{definition}

We also have a forgetful restriction functor
$\mathcal{G}:\mathcal{\rep\, A}\rightarrow \mathcal{C}$ that sends an object $(X,\mu_X)$ to $X$. The induction and restriction functors satisfy Frobenius reciprocity:
\begin{align}
\Hom_{\mathcal{C}}(X,\mathcal{G}(Y)) \cong
\Hom_{\rep A}(\mathcal{F}(X),Y) 
\end{align}
for $X\in\mathcal{C}$ and $Y\in\rep A$.

It was shown in \cite[Theorem 3.12]{CKL}
 that in the module category of a \voa{} $V$ satisfying certain assumptions, certain (super)-algebra objects built from simple currents have (super)-\voa{} structure and give a (super)-\voa{} extension $V_e$ of $V$. If the module category of $V$ is sufficiently nice, then the category $\Rep^{0}V_e$ (with $V_e$ viewed as a categorical (super)-algebra object) is equivalent as a braided tensor category to the category of generalized modules of $V_e$ (with $V_e$ now viewed as a (super)-\voa).  \\

\section{The $\Bp$-algebra as a simple current extension} \label{SectionBp}

Let $\C:=\HR \boxtimes \UHbar{2}\text{-Mod}$ be the Deligne product of $\HR$ (see subsection \ref{Heisenberg}) and $\UHbar{2}$-Mod. The tensor product in this category is given by
\[ (X \boxtimes Y) \otimes (X' \boxtimes Y') = (X \otimes X') \boxtimes (Y \otimes Y') \]
and the braiding, twist, and rigidity morphisms are given by the product of the corresponding morphisms in $\UHbar{2}$-Mod and $\HR$. Notice that for any $\lambda \in i\mathbb{R}$, we have
$$ (\FC{\lambda}{p}) \otimes (\FC{-\lambda}{-p})=\FC{0}{0}. $$
Hence, $\FC{\lambda}{p}$ is a simple current. In what follows, we let $\lambda_p$ satisfy 
\begin{align}
\lambda_p^2=-\dfrac{p}{2}.
\end{align}
We can define an object $\Ap$ of the extended category $\mathcal{C}_{\oplus}$ (which allows infinite direct sums while retaining sufficient structure, see \cite{AR}) by
\begin{align}
\Ap := \bigoplus_{k \in \mathbb{Z}} (\FC{\lambda_p}{p})^{\otimes k} \cong \bigoplus_{k \in \mathbb{Z}} \FC{k\lambda_p}{kp}. 
\end{align}
\begin{remark}
The $\Bp$-algebra is a \voa{} extension of $\Heis\otimes\Mp$ and it decomposes as  $\Heis\otimes\Mp$-module as
\begin{align}
\Bp  \cong \bigoplus_{k \in \mathbb{Z}} \FockC{kp} \boxtimes M_{1-k,1}
\end{align}
so that under the correspondence of Proposition \ref{correspondence} the $\Bp$-algebra is the image of $\Ap$
\begin{align}
\Bp = \bigoplus_{k \in \mathbb{Z}} \FockC{kp} \boxtimes \varphi(\mathbb C^H_{kp}).
\end{align}
\end{remark}

The following is a special case of \cite[Proposition 2.15]{CGR}. Note that the proof of that Proposition is in Appendix A of that paper. 
\begin{proposition}\label{Apalgebra}\textup{\cite[Proposition 2.15]{CGR}}
$\Ap$ can be given a structure of a commutative algebra object in $\mathcal{C}_{\oplus}$ with non-degenerate invariant pairing. This structure is unique up to isomorphism.
\end{proposition}

We now give some criteria on analyzing certain objects in the category $\rep^0 A$ associated to an algebra object $A$.

\begin{lemma}\label{projectivesinduced}
	If $P$ is projective in $\mathcal{C}$, then $\mathcal{F}(P)$ is projective in $\rep A$.
\end{lemma}
\begin{proof}
	Note that
	$\text{Hom}_{\rep A}(\mathcal{F}(P),\bullet)=
	\text{Hom}_{\mathcal{C}}(P,\bullet)\circ\mathcal{G}$
	as functors, due to the Frobenius reciprocity of $\mathcal{F}$ and $\mathcal{G}$.
	Our forgetful restriction functor $\mathcal{G}$ is exact.
	 Also, $\text{Hom}_{\mathcal{C}}(P,\bullet)$ is exact since $P$ is projective.
	Therefore the functor $\text{Hom}_{\rep A}(\mathcal{F}(P),\bullet)$ 
	is exact, which proves that $\mathcal{F}(P)$ is projective.

\end{proof}

\begin{lemma}\label{simplesinduced}
	If $W\cong_{\mathcal{C}}\bigoplus\limits_\nu {F_{\nu}}\boxtimes X_{\nu}$ is a simple object in $\rep A$ then $W$ is isomorphic to the induction of a simple object.
\end{lemma}
\begin{proof}
Pick and fix any $\nu_0$ amongst the $\nu$ appearing above, and pick a non-zero morphism 
$f:S_{\nu_0}\rightarrow X_{\nu_0}$ where $S_{\nu_0}$ is a simple module.
Consider
$g:F_{\nu_0}\boxtimes S_{\nu_0}\xrightarrow{\Id\otimes f} F_{\nu_0}\boxtimes X_{\nu_0}\hookrightarrow \bigoplus\limits_\nu {F_{\nu}}\boxtimes X_{\nu}$ which is a non-zero morphism as well.
By Frobenius reciprocity, we obtain a non-zero morphism
$h:\mathcal{F}(F_{\nu_0}\boxtimes S_{\nu_0})\rightarrow W$.\\

Now, note that $\mathcal{F}(F_{\nu_0}\boxtimes S_{\nu_0})$
as an object of $\mathcal{C}$ decomposes as
$\bigoplus\limits_{k}(F_{k\lambda+\nu_0})\boxtimes (\mathbb{C}_{kp}^H\otimes S_{\nu_0})$,
where the summands are all (mutually inequivalent) simple objects due to the simple current property of $F_{k\lambda}\boxtimes\mathbb{C}_{kp}^H$.
Now, \cite[Theorem 4.4]{CKM} applies and tells us that 
$\mathcal{F}(F_{\nu_0}\boxtimes S_{\nu_0})$ is infact a simple $\rep A$ object. Since $W$ is simple with a non-zero morphism $\mathcal{F}(F_{\nu_0}\boxtimes S_{\nu_0})\rightarrow W$,
$W$ is isomorphic to the induced object $\mathcal{F}(F_{\nu_0}\boxtimes S_{\nu_0})$.

\end{proof}

\begin{theorem} \label{liftingcondition}
The following list of objects in $\C$ induce to $\Rep^{0}(\Ap)$ by the induction functor $\ind:\C \to \Rep(\Ap)$:
\begin{enumerate} 
\item $\FV{\gamma}{\alpha}$ induces to $\Rep^{0}(\Ap)$ if and only if $\alpha+p-1+2\lambda_p \gamma  \in 2\ZZ$; 
\item $\FSC{\gamma}{i}{\ell p}$  induces to $\Rep^{0}(\Ap)$ if and only if $i+\ell p+ 2\lambda_p \gamma  \in 2\mathbb{Z}$;
\item $\FPC{\gamma}{i}{\ell p}$  induces to $\Rep^{0}(\Ap)$ if and only if $i+\ell p  + 2\lambda_p \gamma  \in 2\mathbb{Z}$.
\end{enumerate}
\end{theorem}

\begin{proof}
Recall from subsection \ref{algebracurrents} that given any $\FockH{\gamma} \boxtimes X \in \mathcal{C}^{\oplus}$, $\ind(\FockH{\gamma} \boxtimes X) \in \Rep^{0} \Ap$ iff 
\[\mu_{\ind(\FockH{\gamma} \boxtimes X)} \circ M_{\Ap,\FockH{\gamma} \boxtimes X}= \mu_{\ind(\FockH{\gamma} \boxtimes X)}\]
where $\mu_{\ind(\FockH{\gamma} \boxtimes X)} = (\mu \otimes \Id_{\FockH{\gamma} \boxtimes X}) \circ a^{-1}_{\Ap,\Ap,\FockH{\gamma} \boxtimes X}$ and $M_{A,B}=c_{B,A} \circ c_{A,B}$ is the monodromy. By Proposition \ref{Apalgebra} we can, without loss of generality, assume that $\mu(1_u \otimes 1_v)=1_{u+v}$ in which case the above equation holds iff $M_{\Ap,\FockH{\gamma} \boxtimes X}=\Id$, but by Theorem 2.11 in \cite{CKL} it is enough to check that $M_{\FC{\lambda_p}{p},\FockH{\gamma}\boxtimes X}=\Id$. Note that we have 
\begin{align}
M_{\FC{\lambda_p}{p},\FockH{\gamma} \boxtimes X}= M_{\FockH{\lambda_p}, \FockH{\gamma}} \boxtimes M_{\CC_p^H, X}
\end{align}
where $M_{\FockH{\lambda_p} ,\FockH{\gamma}}=q^{2p \lambda_p \gamma}\Id$ by \eqref{HRbraidtwist}. Recall that the braiding on $\UHbar{2}$-Mod is given by $\tau \circ R$ where $\tau$ is the usual flip map and 
\begin{align}
R=q^{H \otimes H/2} \sum\limits_{n=0}^{p-1}\frac{\{1\}^{2n}}{\{n\}!}q^{n(n-1)/2}E^n \otimes F^n.
\end{align}
Since the generating vector $v_p \in \CC_p^H$ satisfies $Ev_p=Fv_p=0$, the braiding acts as $\tau \circ q^{H \otimes H/2}$ and hence the monodromy $M_{\CC_p^H, X}$  acts as $q^{H\otimes H}$ on $\CC_p^H \otimes X$. The endomorphism rings of $V_{\alpha}$ and $S_i \otimes \CC_{\ell p}^H$ are one dimensional, so the monodromies $M_{\CC_p^H,V_\alpha}$ and $M_{\CC_p^H, S_i \otimes \CC_{\ell p}^H}$  must act as scalars on $\CC_p^H \otimes  V_{\alpha} $ and $\CC_p^H \otimes (S_i \otimes \CC_{\ell p}^H)$, respectively. It follows by direct computation on the usual generators of these modules that
\begin{align}
M_{\CC_p^H, V_{\alpha}}&=q^{p(\alpha+p-1)}\Id,\\
M_{\CC_p^H, (S_i \otimes \CC_{\ell p}^H)}&=q^{p(i+\ell p)}\Id.
\end{align}
It follows that
\begin{align}
M_{\FC{\lambda_p}{p},\FockH{\gamma} \boxtimes V_{\alpha}}&=(-1)^{\alpha+p-1+2\lambda_p \gamma} \Id,\\
M_{\FC{\lambda_p}{p},\FockH{\gamma} \boxtimes (S_i \otimes \CC_{\ell p}^H)}&= (-1)^{i+\ell p + 2\lambda_p \gamma} \Id,
\end{align}
 and so $\FockH{\gamma} \boxtimes V_{\alpha}$ and $\FockH{\gamma} \boxtimes (S_i \otimes \CC_{\ell p}^H)$ lift to $\Rep^0 \Ap$ iff $\alpha+p-1+2\lambda_p \gamma,i+\ell p +2\lambda_p \gamma \in 2\ZZ$, respectively. The endomorphism ring of $P_i$ (and hence the endomorphism ring of $\CC_p^H \otimes (P_i \otimes \CC_{\ell p}^H)$ is two-dimensional spanned by the identity and a nilpotent operator (see Theorem 6.2 in \cite{CGP}), but $M_{\CC_p^H, P_i \otimes \CC_{\ell p}^H}$ has no nilpotent part since it acts by $q^{H \otimes H}$, so it must act by a scalar multiple of the identity. Acting on the vector $v_p \otimes (\sfw_i \otimes v_{\ell p}) $, it is easily seen that 
 \begin{align}
 M_{\CC_p^H, P_i \otimes \CC_{\ell p}^H}=q^{p(i+\ell p)}\Id
 \end{align}
 and so
\[ M_{\FC{\lambda_p}{p}, \FockH{\gamma} \boxtimes (P_i \otimes \CC_{\ell p}^H)}=(-1)^{i+\ell p+2\lambda_p \gamma}\Id .\]
Therefore, $\FockH{\gamma} \boxtimes (P_i \otimes \CC^H_{\ell p})$ lifts iff $i+\ell p +2\lambda_p \gamma \in 2\ZZ$.
\end{proof}

The induction functor $\ind:\mathcal{C}_{\oplus}  \rightarrow \Rep(\Ap)$ is a tensor functor
by Theorem 2.59 in \cite{CKM}
 and $\rep^0\Ap$ is a tensor subcategory of $\rep\Ap$, so for any objects $\ind(U), \ind(V) \in \Rep^{0}(\Ap)$, 
\begin{align}
 \ind(U) \otimes \ind(V) \cong \ind(U \otimes V). 
 \end{align}
$\Rep^{0}(\Ap)$ is rigid by Proposition 2.77 and Lemma 2.78 of \cite{CKM}, and by Proposition 2.67 of \cite{CKM}, the braiding $c_{-,-}^{\Ap}$ on $\Rep^{0}(\Ap)$ satisfies the relation
\begin{align}
\Id_{\Ap} \otimes c_{U,V} =\ind(c_{U,V})=g_{V,U} \circ c^{\Ap}_{\ind (U), \ind (V)} \circ f_{U,V}
\end{align} 
where $f_{U,V}:\ind(U\otimes V)\xrightarrow{\cong}\ind(U)\otimes\ind(V)$ and 
$g_{V,U}:\ind(V)\otimes\ind(U)\xrightarrow{\cong}\ind(V\otimes U)$
are isomorphisms defined in Theorem 2.59 of \cite{CKM}. 
Ultimately, we are interested in the scalars associated to the monodromy isomorphisms $c_{\ind(V),\ind(U)}^{\Ap} \circ c_{\ind(U),\ind(V)}^{\Ap}$, therefore for calculation purposes, we ignore the $f$ and $g$ isomorphisms and simply take
\begin{align}
c_{\ind(U),\ind(V)}^{\Ap}= \Id_{\Ap} \otimes c_{U,V}.
\end{align}  

We compute that $\theta_{(\FockH{\lambda} \boxtimes \mathbb{C}_p^H)^{\otimes k}}=\Id_{(\FockH{\lambda} \boxtimes \mathbb{C}_p^H)^{\otimes k}}$ when $p$ is odd, so $\theta_{\Ap}=\Id_{\Ap}$ for odd p, and hence by Corollary 2.82 and Theorem 2.89 in \cite{CKM}, we have 
\begin{align*}
\theta_{\ind(V)}&=\ind(\theta_V)=\Id_{\Ap} \otimes \theta_V\\
S_{\ind(U),\ind(V)}^{\hopflink}&= \varphi \circ (\Id_{\Ap} \otimes S_{U,V}^{\hopflink}) \circ \varphi^{-1}.
\end{align*}

Using Lemmas \ref{projectivesinduced} and \ref{simplesinduced}, and exactness of the Deligne product and induction functor together with Lemma \ref{prop:SESresolutionofAtypicals}, we easily obtain the following Proposition:

\begin{proposition} \label{RepApstructure}The simple modules in $\Rep^{0}(\Ap)$ are
\begin{align}
 \simTypiBpNAME{\gamma}{\alpha} &= \ind\left(\FV{\gamma}{\alpha}\right)& \text{with } \alpha &\in \CCdot \text{ and } \gamma \lambda_p + \frac{\alpha+p-1}{2} \in \ZZ, \\
 \simAtypiBpNAME{\gamma}{i}{\ell} &= \ind\left(\FSC{\gamma}{i}{\ell p}\right)&  \text{with } i &\in \{0,\dots,p-2\} \text{ and } \gamma \lambda_p + \frac{i+p\ell}{2} \in \ZZ.
\end{align}
We have families of indecomposable modules:
\begin{align}
\projVBpNAME{\gamma}{\alpha} &= \ind\left(\FV{\gamma}{\alpha}\right) &\text{with } \alpha &\notin \CCdot \text{ and } \gamma \lambda_p + \frac{\alpha+p-1}{2} \in \ZZ \; , \\
\projPBpNAME{\gamma}{i}{\ell} &= \ind\left(\FPC{\gamma}{i}{\ell p}\right) \ &\text{with } i &\in \{0,\dots,p-2\} \text{ and } \; \gamma \lambda_p + \frac{i + \ell p}{2} \in \ZZ,
\end{align}
with $\projPBpNAME{\gamma}{i}{\ell}$ being projective, and the above modules satisfy $E^V_{\gamma,\alpha} \cong E^V_{\gamma+k\lambda_p,\alpha+pk},$ $E^S_{\gamma,i,\ell} \cong E^S_{\gamma+k\lambda_p,i,\ell+k},$ $Q^V_{\gamma,\alpha} \cong Q^V_{\gamma+k\lambda_p,\alpha+pk}$, and $Q^P_{\gamma,i,\ell}  \cong Q^P_{\gamma+k\lambda_p,i,\ell+k}$ for all $k \in \mathbb{Z}$. The indecomposable modules admit the following Loewy diagrams:
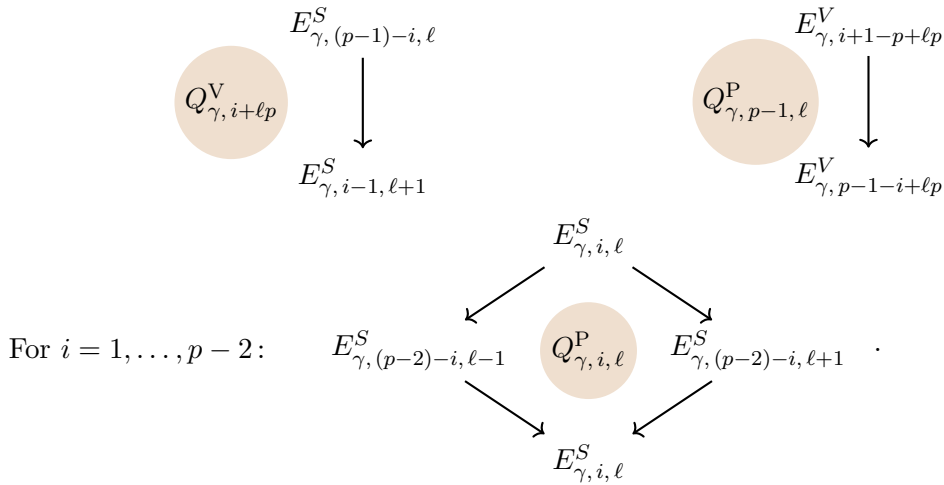
\begin{center}
\begin{tikzpicture}[baseline=(current bounding box.center)]
\node (top) at (1,1) [] {$\simAtypiBpNAME{\gamma}{(p-1)-i}{\ell}$};
\node (bottom) at (1,-1) [] {$\simAtypiBpNAME{\gamma}{i-1}{\ell+1}$};
\draw[->, thick] (top) -- (bottom);
\node (label) at (-0.75,0) [circle, inner sep=2pt, color=black, fill=brown!25!] {$\projVBpNAME{\gamma}{i+\ell p}$};
\end{tikzpicture}
\hspace{3cm}
\begin{tikzpicture}[baseline=(current bounding box.center)]
\node (top) at (1,1) [] {$\simTypiBpNAME{\gamma}{i+1-p+\ell p}$};
\node (bottom) at (1,-1) [] {$\simTypiBpNAME{\gamma}{p-1-i+\ell p}$};
\draw[->, thick] (top) -- (bottom);
\node (label) at (-0.5,0) [circle, inner sep=2pt, color=black, fill=brown!25!] {$\projPBpNAME{\gamma}{p-1}{\ell}$};
\end{tikzpicture}
\\
\begin{tikzpicture}[baseline=(current bounding box.center)]
\node (tag) at (-6,0) [] {\em{For $i = 1, \dots , p-2 \,$:}};
\node (top) at (0,1.5) [] {$\simAtypiBpNAME{\gamma}{i}{\ell}$};
\node (middleI) at (-2.25,0) [] {$\simAtypiBpNAME{\gamma}{(p-2)-i}{\ell-1}$};
\node (middleII) at (2.25,0) [] {$\simAtypiBpNAME{\gamma}{(p-2)-i}{\ell+1}$};
\node (bottom) at (0,-1.5) [] {$\simAtypiBpNAME{\gamma}{i}{\ell}$};
\draw[->, thick] (top) -- (middleI);
\draw[->, thick] (top) -- (middleII);
\draw[->,thick] (middleI) -- (bottom); 
\draw[->, thick] (middleII) -- (bottom);
\node (label) at (0,0) [circle, inner sep=2pt, color=black, fill=brown!25!] {$\projPBpNAME{\gamma}{i}{\ell}$};
\end{tikzpicture} $\; .$
\hspace{3cm} $\text{}$
\end{center}

The category $\Rep^{0}(\Ap)$ is a rigid monoidal category with tensor product $\ind(U) \otimes \ind(V) \cong \ind(U \otimes V).$ $\Rep^{0}(\Ap)$ is also braided with braiding $c^{\Ap}_{\ind(U),\ind(V)}$ defined by
\[
c_{\ind(U),\ind(V)}^{\Ap}= \Id_{\Ap} \otimes c_{U,V},\\
\]
where $c_{U,V}$ is the braiding on $\mathcal{C}_{\oplus}$ given by the product of the braidings on $\mathcal{H}$-Mod and $\UHbar{2}$-Mod. If $p$ is odd, then $\Rep^{0}(\Ap)$ has twist $\theta_{\Ap}$ and Hopf links $S_{\ind(U),\ind(V)}^{\hopflink}$ given by
\begin{align}
\theta_{\ind(V)}&=\Id_{\Ap} \otimes \theta_V,\\
\sS_{\ind(U),\ind(V)}^{\hopflink}&= \sS_{U,V}^{\hopflink} 
\end{align}
where $\theta_V$ and $\sS_{U,V}^{\hopflink}$ are the twist and Hopf links respectively on $\mathcal{C}_{\oplus}$, and we are viewing the Hopf links as scalars.
\end{proposition}

\section{Modular data for typical modules}\label{Sectiontypical}
In this section we compute and compare the modular data for typical and atypical modules in $\Bp$ and $\C$ through the correspondence $\varphi:\UHbar{2}\text{-Mod} \to \Mp\text{-Mod}, V_{\alpha} \mapsto \FockS{\frac{\alpha+p-1}{\sqrt{2p}}}, S_i \otimes \mathbb{C}^H_{\ell p} \mapsto M_{1-\ell,i+1}$ of Proposition \ref{correspondence} found in \cite{CMR}.

\subsection{Typical modules}
The typical modules in $\C$ take the form $\ind(\FV{\gamma}{\alpha})$ for $\alpha \in \CCdot:=(\CC\setminus \ZZ)\ \cup p\ZZ$, which is associated to the $\Bp$-Module 
\begin{align}
E_{\gamma,\alpha}:=\mathscr{F}\left(\FockC{\gamma} \boxtimes \FockS{\frac{\alpha+p -1}{\sqrt{2p}}}\right)
 \end{align}
through the above correspondence. Recall that $\FockC{\gamma}$ denotes the usual Fock space as a module of the Heisenberg VOA $\Heis$, and $\FockS{\frac{\alpha+p -1}{\sqrt{2p}}}$ as a module of the singlet VOA $\Mp$. Note that by Theorem \ref{liftingcondition}, $E_{\gamma,\alpha}\in\rep^0 \Bp$ iff $\gamma\lambda_p+\frac{\alpha}{2} + \frac{{p-1}}{2} \in\ZZ$ and that since $\gamma$ takes purely imaginary values, we are therefore forced to take {$\alpha\in(\RR\setminus\ZZ)\cup p\ZZ$}. On the quantum group side,  $\FockH{\lambda_p} \boxtimes \CC^H_{p}$ induces to the algebra object $\Ap$ under $\ind: \C \to \Ap\text{-Mod}$, and on the VOA side, $\FockC{\lambda_p}\otimes M_{0,1}$
	induces to $\Bp$. Therefore, we have 
\begin{align}
E_{\gamma,\alpha}\cong E_{\gamma+\lambda_p,\alpha+p}.
\end{align}
We re-parameterize the space of simple typicals with the following substitution:
\begin{equation}
\nu=\frac{2\alpha}{p},\quad \ell = \gamma\lambda_p+\frac{\alpha}{2}.\label{eqn:nuellnotation}
\end{equation}
Considering isomorphisms, our parameter space reduces to $\ell+\frac{p-1}{2}\in\ZZ, {\nu\in(-1,1]\backslash \frac{2}p\ZZ}$.
This re-parametrization is done to facilitate a comparison with \cite{CR2} and \cite{CR3} which correspond to $p=2$ and $p=3$ cases. The following $\sS^\chi$ matrix calculations will also work with $\nu\in(-1,1]$, and therefore, for what follows we work with this slightly larger parameter space.
By abuse of notation, we still let $E_{\nu,\ell}:=E_{\gamma,\alpha}$.  For characters, we use the following convention:
\begin{align}
\text{ch}[\FockC{\gamma}]=\frac{z^{-4\lambda_p\gamma/p} q^{\gamma^2/2}}{\eta(q)}
\end{align}
where $z=e^{2\pi i \zeta}$, $q=e^{2\pi i \tau}$. We need another variable $y=e^{2 \pi i \kappa}$ in the characters to have the $\sS$ matrix come out correctly.
We multiply the characters by $y^{-4/p}=e^{-8\pi i \kappa/p}$.
Under the $\sS$-transformations, we have:
\begin{align}
\sS: (\zeta,\tau,\kappa) \mapsto \left(\frac{\zeta}{\tau},-\frac{1}{\tau},\kappa-\frac{\zeta^2}{\tau}\right).
\end{align}

\begin{lemma} \label{lem:CHARatypical}
Using the parametrisation \eqref{eqn:nuellnotation} for the typical modules, their (super)characters are given by
\begin{align}
 \ch{E_{\nu,\ell}}{z;q} &= \frac{e^{-8\pi i \kappa/p}}{\eta(\tau)^2}\sum_{m\in\ZZ}e^{2\pi i \tau\ell^2/p}e^{\pi i (\nu -4\ell/p)m}\delta(2\zeta+\ell\tau-m) 
 \quad\quad p \,\,\mathrm{odd}\\
 \sch{E_{\nu,\ell}}{z;q} &= \frac{e^{-8\pi i \kappa/p}}{\eta(\tau)^2}\sum_{m\in\ZZ + (1/2)}e^{2\pi i \tau\ell^2/p}e^{\pi i (\nu -4\ell/p)m}\delta(2\zeta+\ell\tau-m) 
 \quad\quad p \,\,\mathrm{even}.
\end{align}
\end{lemma}
\begin{proof}
Let $p$ be odd. 
By construction, we have:
\begin{align}
\ch{E_{\gamma,\alpha}}{z;q} &=\sum_{k \in \mathbb{Z}}
e^{-8\pi i \kappa/p}\text{ch}\left[\FockH{\gamma+k\lambda_p}\boxtimes \FockS{\frac{\alpha+(k+1)p -1}{\sqrt{2p}}}\right]\nonumber\\
&=\sum_{k\in\mathbb{Z}}
\dfrac{e^{-8\pi i \kappa/p}z^{-4\gamma\lambda_p/p-4k\lambda_p^2/p}q^{(\gamma+k\lambda_p)^2/2}}{\eta(\tau)}
\dfrac{q^{\frac{1}{2}\left(  \frac{\alpha+(k+1)p-1}{\sqrt{2p}} -\frac{\alpha_0}{2}\right)^2}}{\eta(\tau)}\nonumber\\
&= \dfrac{e^{-8\pi i \kappa/p}z^{-4\gamma\lambda_p/p}q^{\frac{\gamma^2}{2}+\frac{\alpha^2}{4p}}}{\eta(\tau)^2}\sum_{k\in\ZZ}
z^{2k}q^{k(\gamma\lambda_p  + \frac{\alpha}{2})}\nonumber\\
&=\dfrac{e^{-8\pi i \kappa/p}e^{2\pi i \zeta(-4\gamma\lambda_p/p)}e^{2\pi i \tau\left(\frac{\gamma^2}{2}+\frac{\alpha^2}{4p}\right)}}{\eta(\tau)^2}
\sum_{m\in\ZZ}\delta(2\zeta + (\gamma\lambda_p+\alpha/2)\tau -m),
\end{align}
where our $\delta$-function is supported at $0$.
Passing to the $\nu, \ell$ notation \eqref{eqn:nuellnotation}, we observe the following relations:
\begin{align}
-4\lambda_p\gamma\zeta/p&=  \zeta\nu  - 4\zeta\ell/p, \\
\gamma^2/2 + \alpha^2/4p &= -\ell^2/p + \nu\ell/2.
\end{align}
The expression for $\ch{E_{\nu,\ell}}{z;q}$ follows straightforwardly.
For $p$ even, we introduce a factor of $e^{\pi i k}$ in the very first summation. The rest now follows similarly.
\end{proof}

If $p$ is odd, we set 
\begin{align}
\sS\left\lbrace\text{ch}[E_{\nu,\ell}]\right\rbrace&=
\frac{e^{-8\pi i (\kappa-\zeta^2/\tau)/p}}{\eta(-1/\tau)^2}\sum_{m\in\ZZ}e^{-2\pi i \ell^2/p\tau}e^{\pi i (\nu -4\ell/p)m}\delta\left(\frac{2\zeta-\ell-m\tau}{\tau}\right)\nonumber\\
&=\frac{|\tau|}{-i \tau\eta(\tau)^2}e^{-8\pi i (\kappa-\zeta^2/\tau)/p}\sum_{m\in\ZZ}e^{-8\pi i\zeta^2/p\tau} e^{2\pi i m^2\tau/p}e^{\pi i \nu m}\delta\left({2\zeta-\ell-m\tau}\right)\nonumber\\
&=\frac{|\tau|}{-i \tau\eta(\tau)^2}e^{-8\pi i \kappa/p}\sum_{m\in\ZZ} e^{2\pi i m^2\tau/p}e^{\pi i \nu m}\delta\left({2\zeta-\ell-m\tau}\right)
\end{align}
and similarly if $p$ is even,
\begin{align}
\sS\left\lbrace\text{sch}[E_{\nu,\ell}]\right\rbrace
&=\frac{|\tau|}{-i \tau\eta(\tau)^2}e^{-8\pi i \kappa/p}\sum_{m\in\ZZ + (1/2)} e^{2\pi i m^2\tau/p}e^{\pi i \nu m}\delta\left({2\zeta-\ell-m\tau}\right).
\end{align}
Regardless of the parity of $p$, let: 
\begin{equation} \label{SchiTYPICALS}
\sS^{\chi}_{(\nu,\ell),(\nu',\ell')}=\frac{|\tau|}{-i\tau}\frac{1}{2}e^{\pi i\left( 4\ell\ell'/p -\ell\nu' - \ell'\nu\right)}.	
\end{equation}
Now we show that this $\sS$-matrix correctly gets us the $\sS$-transformations of the characters in the $p$ odd case.
In the $p$ even case, the calculation is similar, except with characters replaced with supercharacters and 
summations over $\ZZ$ now changed to summations over $\ZZ+(1/2)$.
\begin{align}
&\sum_{\ell'\in\ZZ}\int_{-1}^1\sS^{\chi}_{(\nu,\ell),(\nu',\ell')}\text{ch}[E_{\nu',\ell'}]d\nu'\nonumber\\
&=\sum_{\ell'\in\ZZ}\int_{-1}^1\frac{|\tau|}{-i\tau}\frac{1}{2}e^{\pi i\left( 4\ell\ell'/p -\ell\nu' - \ell'\nu\right)}
\frac{e^{-8\pi i \kappa/p}}{\eta(q)^2}\sum_{m\in\ZZ}e^{2\pi i \tau\ell'^2/p}e^{\pi i (\nu' -4\ell'/p)m}\delta(2\zeta+\ell'\tau-m)d\nu'\nonumber\\
&=\sum_{\ell'\in\ZZ}\frac{|\tau|}{-i\tau}e^{\pi i\left( 4\ell\ell'/p  - \ell'\nu\right)}
\frac{e^{-8\pi i \kappa/p}}{\eta(q)^2}\sum_{m\in\ZZ}e^{2\pi i \tau\ell'^2/p}e^{\pi i (-4\ell'/p)m}\delta(2\zeta+\ell'\tau-m)\int_{-1}^1
\frac{1}{2}e^{\pi i (\nu'm -\nu'\ell)}d\nu'\nonumber\\
&=\sum_{\ell'\in\ZZ}\frac{|\tau|}{-i\tau}e^{\pi i\left( 4\ell\ell'/p  - \ell'\nu\right)}
\frac{e^{-8\pi i \kappa/p}}{\eta(q)^2}\sum_{m\in\ZZ}e^{2\pi i \tau\ell'^2/p}e^{\pi i (-4\ell'/p)m}\delta(2\zeta+\ell'\tau-m)\delta(\ell-m)\nonumber\\
&=\sum_{\ell'\in\ZZ}\frac{|\tau|}{-i\tau}e^{\pi i\left(- \ell'\nu\right)}
\frac{e^{-8\pi i \kappa/p}}{\eta(q)^2}e^{2\pi i \tau\ell'^2/p}\delta(2\zeta+\ell'\tau-\ell)\nonumber\\
&=\sum_{m\in\ZZ}\frac{|\tau|}{-i\tau\eta(q)^2}e^{-8\pi i \kappa/p}e^{\pi i m\nu}
e^{2\pi i \tau m^2/p}\delta(2\zeta-m\tau-\ell)\nonumber\\
&=\sS^{\chi}\left\lbrace\text{ch}[E_{\nu,\ell}]\right\rbrace. 
\end{align}

All that remains now is to make an appropriate choice of normalisation. For this purpose, we obtain a resolution of the tensor identity $S_0 \otimes \mathbb{C}_0^H$, and by induction, a resolution of $\Ap$ in terms of other typical modules through the short exact sequences from Lemma \ref{prop:SESresolutionofAtypicals}. We transfer this resolution to the \voa{} side and obtain a character relation for $\Bp$ from Proposition \ref{correspondence}.\\

Shifting the short exact sequence in Lemma \ref{prop:SESresolutionofAtypicals} by $(j,\ell)\to(j+1,\ell-1)$, we can obtain a family of linked short exact sequences
\begin{equation}
0 \rightarrow S_{(p-2)-j_k} \otimes \mathbb{C}_{(\ell_k-1)p}^H \rightarrow V_{(j_k+1)+ (\ell_k-1)p} \xrightarrow{f_k} S_{j_k} \otimes \mathbb{C}_{\ell_k p}^H \rightarrow 0,
\end{equation}

where $(j_{k+1},\ell_{k+1})=(p-2-j_k,\ell_k-1)$. We therefore obtain a long exact sequence

\begin{equation} \label{preRESOLUTIONtensorid}
\cdots \rightarrow  V_{n_4} \xrightarrow{f_4} V_{n_3} \xrightarrow{f_3} V_{n_2} \xrightarrow{f_2} V_{n_1} \xrightarrow{f_1} V_{n_0} \xrightarrow{f_0}  S_{j_0} \otimes \CC^H_{\ell_0 p} \rightarrow  0
\end{equation}
where $n_k=(j_k+1)+(\ell_k-1)p$. We can take $(j_0,\ell_0)=(0,0)$, tensor this sequence with $\FockH{0}$, and apply the induction functor (see Definition \ref{inductionfunctor}) and lastly, the correspondence in Proposition \ref{correspondence}. This sends $S_0 \otimes \mathbb{C}^H_{0}$ to $\Bp$, so we obtain the following relation for the character of $\Bp$:
\begin{equation} \label{chBprelation} \mathrm{ch}[\Bp] = \sum_{m=0}^\infty (-1)^m \mathrm{ch}[Y_{n_m}] = \sum_{m=0}^\infty \big(\mathrm{ch}[Y_{n_{2m}}] - \mathrm{ch}[Y_{n_{2m+1}}]\big), \end{equation}
where $Y_{n_m}=\ind(\FockC{0} \boxtimes \FockS{\frac{n_m+p-1}{\sqrt{2p}}})$ (recall that $\FockS{\frac{n_m+p-1}{\sqrt{2p}}}$ corresponds to $V_{n_m}$ through proposition \ref{correspondence}). Note that by Theorem \ref{liftingcondition}, the module $\FockC{0} \boxtimes V_{n_m}$ lifts to $\Rep^{0}(\Ap)$ through the induction functor iff $n_m+p-1 \in 2\ZZ$.  Observe by a simple induction that with $(j_0,\ell_0)=(0,0)$, the indices $n_m$ satisfy
\begin{equation}
n_m = \left\{\begin{array}{cl} 1-(m+1)p & \text{for } m \text{ even} \\
-1 - mp & \text{for } m \text{ odd},
\end{array}\right.  
\end{equation}
so $\FockH{0} \boxtimes V_{n_m}$ lifts to $\Rep^{0} \Ap$ 
and $\FockC{0}\boxtimes \FockS{\frac{n_m+p-1}{\sqrt{2p}}}$ lifts to $\Rep^{0}\Bp$
for all $m$ and $p$. Adopting the paramertisation \eqref{eqn:nuellnotation} for the modules $Y_{n_m}$ gives the following:
\begin{equation}
Y_{n_m} = \left\{\begin{array}{cl} E_{\frac{2(1-(m+1)p)}{p},\frac{1-(m+1)p}{2}} & \text{for } m \text{ even} \\
E_{\frac{2(-1-mp)}{p},\frac{-1-mp}{2}} & \text{for } m \text{ odd}.
\end{array}\right. 
\end{equation}
We write $\mathds{1}$ for the $\Bp$-algebra 
and obtain from \eqref{chBprelation} the relation
\begin{equation} \label{niceSseries}
\sS^\chi_{\mathds{1},(\nu',\ell')} = \sum_{m = 0}^\infty \left( \sS^\chi_{\left(\frac{2(1-(2m+1)p)}{p},\frac{1-(2m+1)p}{2}\right),(\nu',\ell')} - \; \sS^\chi_{\left(\frac{2(-1-(2m+1)p)}{p},\frac{-1-(2m+1)p}{2}\right),(\nu',\ell')} \right).
\end{equation}
By substituting \eqref{SchiTYPICALS} into \eqref{niceSseries} and simplifying, we deduce that
\begin{align*}
\sS^\chi_{\mathds{1},(\nu',\ell')} &= \frac{|\tau|}{-i\tau}\frac{1}{2}\sum_{m = 0}^\infty \left( e^{-\pi i \frac{1-(2m+1)p}{2}\nu'} - \; e^{-\pi i \frac{-1-(2m+1)p}{2}\nu'} \right)\nonumber\\
&= \frac{|\tau|}{-i\tau}\frac{1}{2}\sum_{m = 0}^\infty \left( e^{\pi i (p-1) \frac{\nu'}{2}}e^{\pi i \frac{2mp}{2}\nu'} - \; e^{\pi i (p+1) \frac{\nu'}{2}}e^{\pi i \frac{2m p}{2}\nu'} \right).
\end{align*}
Setting $x = e^{\pi i \frac{\nu'}{2}}$, we see that
\begin{align}
\sS^\chi_{\mathds{1},(\nu',\ell')} &= \frac{|\tau|}{-i\tau}\frac{1}{2} \left( x^{p-1} \sum_{m = 0}^\infty x^{2mp} - x^{p+1} \sum_{m = 0}^\infty x^{2m p} \right).\nonumber \\
&=\frac{|\tau|}{-i \tau} \frac{x^p}{2}(x^{-1}-x) \sum\limits_{m=0}^{\infty} x^{2mp}
\end{align}

Since $x$ lies on the unit circle this sum is not convergent. If we infinitesimally deform $x$ to lie within the unit circle, this sum converges to 
\begin{align}
 \frac{|\tau|}{-i\tau}\frac{1}{2} \cdot \frac{x-x^{-1}}{x^p-x^{-p}} \;  .
\end{align}

This is the motivation for our choice of normalisation factor $\sS^{\chi}_{\mathds{1}, (\nu',\ell')}$, i.e. we define it to be
\begin{align}
\sS^\chi_{\mathds{1},(\nu',\ell')} := \frac{|\tau|}{-i\tau}\frac{1}{2} \cdot \frac{x-x^{-1}}{x^p-x^{-p}}, \label{Schinormalisation} 
\end{align}
where $x = e^{\pi i \frac{\nu'}{2}}$. As a further motivation we note that this type of regularization worked very well in the Verlinde formula story of \cite{CR2, CR3}.

We will now compare the modular $\sS^{\chi}$ matrix coming from modular transformations of characters for $\Bp$ with the $\sS^{\hopflink}$ matrix coming from Hopf links in $\C$ for typical modules calculated on the quantum group side. 
These quantities agree up to normalised conjugation.

\begin{proposition} \label{typicalcomparison}The normalized character matrix $\sS^{\chi}$ and Hopf link $\sS^{\hopflink}$ agree up to conjugation. That is, 
\begin{align}
\frac{\sS^{\chi \ast}_{(\nu, \ell),(\nu ', \ell ')}}{\sS^{\chi \ast}_{ \mathds{1},(\nu ', \ell ')}} = \frac{\sS^{\hopflink}_{(\nu,\ell),(\nu ', \ell ')}}{\sS^{\hopflink}_{\mathds{1},(\nu ', \ell ')}},
\end{align}
where $\ast$ denotes complex conjugation of the entries.
\end{proposition}
\begin{remark} \label{remark:cc} There are two ways to explain the appearance of complex conjugation
\begin{enumerate}
\item Braiding is a structure that we give our tensor category and we could equally well have chosen the reversed braiding.
In our case this would have amounted to have chosen the inverse of the $R$-matrix as $R$-matrix and the effect is the same as replacing $q$ by $q^{-1}$, that is Hopf links with reversed-braiding are precisely the complex conjugates of our Hopf links. 
\item In a rigid braided tensor category $\mathcal C$ with tensor identity $\mathds 1$ the map 
\[
\mathcal G(\mathcal C) \rightarrow \mathbb C, \qquad X \mapsto \frac{S^\hopflink_{X, Y}}{S^\hopflink_{\mathds 1, Y}}
\]
is a ring homomorphism from the Grothendieck ring of the category to $\mathbb C$ for any simple object $Y$. Composing with complex conjugation of course doesnot change the homomorphism property so that the Proposition can be rephrazed that normalied columns of the modular $S$-matrix give one-dmiensional representations of the Grothendieck ring.
\end{enumerate}
\end{remark}
\begin{proof}
It follows from equation \ref{SchiTYPICALS} and Definition \ref{Schinormalisation} that the normalisation of $\sS^{\chi}$ is given by
\begin{equation}
\frac{\sS^\chi_{(\nu,\ell),(\nu',\ell')}}{  \sS^\chi_{\mathds{1},(\nu',\ell')}} = \frac{x^p-x^{-p}}{x-x^{-1}} \cdot e^{\pi i\left( 4\ell\ell'/p -\ell\nu' - \ell'\nu\right)} 
\end{equation}
where $x = e^{\pi i \nu' / 2}$. 
Since the $\sS^{\hopflink}$ matrix is preserved under induction, we just calculate it in the category $\mathcal{C}$.
The $\sS^{\hopflink}$-matrix for typical modules satisfies the relation 
\begin{equation}
\sS^{\hopflink}_{\FockH{\gamma_1} \boxtimes V_{\alpha_1},\FockH{\gamma_2} \boxtimes V_{\alpha_2}}=\sS^{\hopflink}_{\FockH{\gamma_1},\FockH{\gamma_2}}\sS^{\hopflink}_{V_{\alpha_1},V_{\alpha_2}}
\end{equation}
where $\sS^{\hopflink}_{\FockH{\gamma_1},\FockH{\gamma_2}}=e^{2\pi i \gamma_1\gamma_2}$. By Lemma 6.6 in \cite{CGP}, we have 
\begin{align}
\sS^{\hopflink}_{V_{\alpha_1},V_{\alpha_2}}=(-1)^{p-1}pq^{\alpha_1\alpha_2},\qquad \sS^{\hopflink}_{S_0,V_{\alpha_2}}=(-1)^{p-1}\frac{p\{\alpha_2\}}{\{p\alpha_2\}}.
\end{align}
 Hence,
\begin{align}
\frac{\sS^{\hopflink}_{\FockH{\gamma_1} \boxtimes V_{\alpha_1},\FockH{\gamma_2} \boxtimes V_{\alpha_2}}}{\sS^{\hopflink}_{\FockH{0} \boxtimes S_0,\FockH{\gamma_2} \boxtimes V_{\alpha_2}}}=e^{2 \pi i \gamma_1 \gamma_2}q^{\alpha_1\alpha_2}\frac{\{p\alpha_2\}}{\{\alpha_2\}}. 
\end{align}
Setting $(\nu_i,\ell)=(\frac{2\alpha_1}{p},\lambda\gamma_1+\frac{\alpha_1}{2})$, $(\nu', \ell ')=(\frac{2\alpha_2}{p},\lambda\gamma_2+\frac{\alpha_2}{2})$, $x=e^{\pi i \nu'/2}$, and adopting the notation $\sS^{\hopflink}_{(\nu,\ell),(\nu ', \ell')}=\sS^{\hopflink}_{\FockH{\gamma_1} \boxtimes V_{\alpha_1},\FockH{\gamma_2} \boxtimes V_{\alpha_2}}$, $\mathds{1}=\FockH{0} \boxtimes S_0$, we easily see that
\[\frac{\sS^{\hopflink}_{(\nu,\ell),(\nu ' \ell ')}}{\sS^{\hopflink}_{\mathds{1},(\nu ', \ell ')}}=\frac{x^p-x^{-p}}{x-x^{-1}} \cdot  e^{- \pi i(4 \ell \ell'/p-\ell \nu '-\ell' \nu)}.\]
\end{proof}

\section{Modular data for atypical modules}\label{Sectionatypical}

Recall that by $\mathcal{G}^{\text{ss}}$ we mean the quotient of the Grothedieck ring by the ideal of the Grothendieck ring formed by the negligible objects.

In this section we derive $\mathcal{G}^{\text{ss}}(\mathcal{C}^{0})$ for the category of local modules $\mathcal{C}^{0}$ in $\mathcal{C}$ (those which are induced to $\Rep^{0} \Ap$ by the induction functor). We also compute the corresponding $\mathsf{S}^{\hopflink}$-matrix and make a comparison with the matrix $\sS^{\chi}$ coming from modular transformations of characters appearing in the Verlinde algebra of characters $\mathcal{V}^{\text{ss}}(\Bp)$ of the semisimplification of $\Bp$-Mod. The Verlinde formula then follows from the standard categorical argument.

\subsection{Structure of Grothendieck rings} \label{sec:grothendieck}
We start by determining structure of the underlying Grothendieck rings.
\begin{proposition}\label{Grothendieck}
$\mathcal{G}^{\text{ss}}(\mathcal{C}^{0})$ has minimal generating set with elements $\module{\gamma}{i}{p\ell}$ of the form
\begin{align}
\module{\frac{2n}{2\lambda_p}}{2m}{0}, & \qquad n \in \{0,1,\dots,\frac{p-1}{2}\}, m\in\{0,1,\dots,\frac{p-3}{2}\},\\
\module{\frac{2n+1}{2\lambda_p}}{2m}{p},& \qquad n\in \{0,1,\dots,\frac{p-3}{2}\}, m \in \{0,1,\dots,\frac{p-3}{2}\},
\end{align}
when p is odd, and
\begin{align}
\module{\frac{2n}{2\lambda_p}}{2m}{0}, & \qquad n \in \{0,1,\dots,\frac{p-2}{2}\}, m\in\{0,1,\dots,\frac{p-2}{2}\},\\
\module{\frac{2n+1}{2\lambda_p}}{2m+1}{0},& \qquad n \in \{0,1,\dots,\frac{p-2}{2}\}, m \in \{0,1,\dots,\frac{p-4}{2}\},
\end{align}
when p is even. The product is given by 
\begin{equation}\label{Grothendieckproduct}\module{\frac{x}{2\lambda_p}}{i}{p\ell} \mdot \module{\frac{y}{2\lambda_p}}{j}{p\ell '}\left\{=\begin{array}{cc}
\sum\limits_{\substack{l=|i-j| \\ \text{by } 2}}^{i+j} \module{\frac{x+y \text{ mod p}}{2\lambda_p}}{l}{p(\ell+\ell')\text{ mod 2p}}  & \text{ if $i+j<p$},\\
&\\
\sum\limits_{\substack{l=|i-j| \\ \text{by } 2}}^{2p-4-i-j}\module{\frac{x+y \text{ mod p}}{2\lambda_p}}{l}{p(\ell+\ell')\text{ mod 2p}}  & \text{if $i+j \geq p$}. 
\end{array}\right.
\end{equation}
The coefficients of the sum corresponding to a product of minimal generators are positive, giving a $\mathbb{Z}_+$-basis.
\end{proposition}

\begin{proof}
We first note that the quantum dimension of an object $V \in \UHbar{2}$-Mod is given by 
\[\sum_{i=0}^k v_i^*(K^{1-p}v_i), \]
where $\{v_0,...,v_k\}$ is a basis for V. Hence, if $V=V_{\alpha}$, then
\begin{align}
\qdim(V_{\alpha})=\sum\limits_{i=0}^{p-1}v_i^*(K^{1-p}v_i)=\sum\limits_{i=0}^{p-1}q^{(p-1)(\alpha+p-1-2i)}=q^{(p-1)(\alpha+p-1)}\frac{1-q^{2p}}{1-q^2}=0. 
\end{align}
By Lemma \ref{prop:SESresolutionofAtypicals} we have the short exact sequence
\[0 \rightarrow V_{r-1-i+\ell r} \rightarrow P_i  \otimes \Ck \rightarrow V_{1+i-r+\ell r} \rightarrow 0,\]
so we see that the quantum dimension of the $P_i \otimes \mathbb{C}_{\ell p}^H$ is also zero. Hence, $\mathcal{G}^{\text{ss}}(\mathcal{C}^{0})$ is generated by the elements corresponding to the $\FSC{\gamma}{i}{\ell p}$ which are induced into $\Rep^{0}(\Ap)$ by the induction functor, which by theorem \ref{liftingcondition}, are those objects such that $i+p \ell +2\lambda_p \gamma \in 2\mathbb{Z}$. We will denote the object in $\mathcal{G}^{\text{ss}}(\mathcal{C}^{0})$corresponding to $\FSC{\gamma}{i}{\ell p}$ by $\module{\gamma}{i}{p\ell}$. Recall from Lemma \ref{prop:SESresolutionofAtypicals} that for any $i\in \{1,...,p-1\}$ we have the short exact sequence
\[ 0 \rightarrow S_{p-1-i} \otimes \Ck \rightarrow V_{i+\ell p} \rightarrow S_{i -1} \otimes \mathbb{C}_{p(\ell+1)}^H \rightarrow 0,\]
so we have $\module{\gamma}{p-1-i}{p\ell}=-\module{\gamma}{i-1}{p(\ell+1)}$. So, for any $i \in \{0,...,p-2\}$, we have
\begin{align}
\module{\gamma}{i}{p\ell}=\module{\gamma}{p-1-(p-1-i)}{p\ell}=-\module{\gamma}{p-1-(i+1)}{p(\ell+1)}=\module{\gamma}{i}{p(\ell+2)}.
\end{align}
Hence,  $\mathcal{G}^{\text{ss}}(\mathcal{C}^{0})$ is generated by elements of the form $\module{\gamma}{i}{0}$ and $\module{\gamma}{i}{p}$ which satisfy
\begin{align}
i+2\lambda_p\gamma &\in 2\mathbb{Z}, \label{eqn1}\\
i+p+2\lambda_p\gamma &\in 2\mathbb{Z}.\label{eqn2}
\end{align}
If $i$ is odd in \eqref{eqn1}, then $i+2\lambda_p \gamma$ being even implies that $\gamma=\frac{2n+1}{2\lambda_p}$ for some $n \in \mathbb{Z}$, and $i$ even in \eqref{eqn1} implies $\gamma=\frac{2n}{2\lambda_p}$ for some $n \in \mathbb{Z}$. If $p$ and $i$ are both odd or even in \eqref{eqn2}, then $i+p$ is even, so $\gamma=\frac{2n}{2\lambda_p}$ for some $n \in \mathbb{Z}$. If one of $p$ and $i$ is odd in \eqref{eqn2}, and one is even, then $i+p$ is odd, so $\gamma=\frac{2n+1}{2\lambda_p}$ for some $n \in \mathbb{Z}$. Notice also that $\FC{\lambda_p}{p}$ induces to the identity in $\Rep^{0} \Ap$ and hence corresponds with the identity in $\mathcal{C}^{0}$. So, we have the relation
\begin{align}
 \module{\gamma}{i}{p\ell}=\module{\gamma}{i}{p\ell} \mdot \module{\lambda_p}{0}{p}=\module{\gamma+\lambda_p}{i}{p(\ell+1)}, 
 \end{align}
and $\frac{p}{2\lambda_p}+\lambda_p=0$, so the elements $\module{\gamma}{i}{p\ell}$ have $\gamma\in \{0,\frac{1}{2\lambda_p},...,\frac{p-1}{2\lambda_p}\}$. Combining the above remarks, we see that when $p$ is odd, $\mathcal{G}^{\text{ss}}(\mathcal{C}^{0})$ is generated by the objects
\begin{align*}
\module{\frac{2n}{2\lambda_p}}{2m}{0}, & \qquad n \in \{0,1,\dots,\frac{p-1}{2}\}, m\in\{0,1,\dots,\frac{p-3}{2}\}\\
\module{\frac{2n}{2\lambda_p}}{2m+1}{p}, & \qquad n \in \{0,1,\dots,\frac{p-1}{2}\}, m \in \{0,1,\dots,\frac{p-3}{2}\}\\
\module{\frac{2n+1}{2\lambda_p}}{2m+1}{0},& \qquad n \in \{0,1,\dots,\frac{p-3}{2}\}, m \in \{0,1,\dots,\frac{p-3}{2}\}\\
\module{\frac{2n+1}{2\lambda_p}}{2m}{p},& \qquad n\in \{0,1,\dots,\frac{p-3}{2}\}, m \in \{0,1,\dots,\frac{p-3}{2}\}
\end{align*}
and if $p$ is even, then $\mathcal{G}^{\text{ss}}(\mathcal{C}^{0})$ is generated by the objects
\begin{align*}
\module{\frac{2n}{2\lambda}}{2m}{0}, & \qquad n \in \{0,1,\dots,\frac{p-2}{2}\}, m\in\{0,1,\dots,\frac{p-2}{2}\}\\
\module{\frac{2n}{2\lambda}}{2m}{p},& \qquad n\in \{0,1,\dots,\frac{p-2}{2}\}, m \in \{0,1,\dots,\frac{p-2}{2}\}\\
\module{\frac{2n+1}{2\lambda}}{2m+1}{0},& \qquad n \in \{0,1,\dots,\frac{p-2}{2}\}, m \in \{0,1,\dots,\frac{p-4}{2}\}\\
\module{\frac{2n+1}{2\lambda}}{2m+1}{p}, & \qquad n \in \{0,1,\dots,\frac{p-2}{2}\}, m \in \{0,1,\dots,\frac{p-4}{2}\}
\end{align*}
The tensor products for the $\FockH{\gamma}$ and $S_i$ modules (see \cite[Proposition 8.4]{CGP}) are given by 
\begin{align*}
\FockH{\gamma} \otimes \FockH{\gamma'}=\FockH{\gamma+\gamma'} \qquad \text{and} \qquad
S_i \otimes S_j= \left\{ 
\begin{array}{cc}
\bigoplus\limits_{\substack{l=|i-j| \\ \text{by } 2}}^{i+j} S_l  & \text{ if $i+j<p$},\\
&\\
\bigoplus\limits_{\substack{l=|i-j| \\ \text{by } 2}}^{2p-4-i-j} S_l \oplus \bigoplus\limits_{\substack{l=2p-2-i-j \\ \text{by } 2}}^{r-1} P_l & \text{if $i+j \geq p$}. 
\end{array}
\right.
\end{align*}
The projective indecomposable modules denoted by $P_{i}$ have quantum dimension zero so their corresponding object in $\mathcal{G}^{\text{ss}}(\mathcal{C}^{0})$ is zero. We therefore have the product
\begin{equation}\label{Grothendieckproduct}\module{\frac{x}{2\lambda}}{i}{p\ell} \mdot \module{\frac{y}{2\lambda}}{j}{p\ell '}\left\{=\begin{array}{cc}
\sum\limits_{\substack{l=|i-j| \\ \text{by } 2}}^{i+j} \module{\frac{x+y \text{ mod p}}{2\lambda}}{l}{p(\ell+\ell')\text{ mod 2p}}  & \text{ if $i+j<p$},\\
&\\
\sum\limits_{\substack{l=|i-j| \\ \text{by } 2}}^{2p-4-i-j}\module{\frac{x+y \text{ mod p}}{2\lambda}}{l}{p(\ell+\ell')\text{ mod 2p}}  & \text{if $i+j \geq p$}. \\
\end{array}\right.\\
\end{equation}
This is not a minimal generating set since we have not yet accounted for the relation $\module{\gamma}{p-1-i}{p\ell}=-\module{\gamma}{i-1}{p(\ell+1)}$. Consider the case $p$ odd. Notice that every generator of the form $\module{\frac{2n}{2\lambda}}{2m}{0}$ corresponds to a generator of the form $\module{\frac{2n}{2\lambda}}{2m+1}{p}$, and each generator of the form $\module{\frac{2n+1}{2\lambda}}{2m+1}{0}$ to one of the form $\module{\frac{2n+1}{2\lambda}}{2m}{p}$ under $\module{\gamma}{p-1-i}{p\ell}=-\module{\gamma}{i-1}{p(\ell+1)}$. Hence, the set
\begin{align*}
\module{\frac{2n}{2\lambda}}{2m}{0}, & \qquad n \in \{0,...,\frac{p-1}{2}\}, m\in\{0,...,\frac{p-3}{2}\}\\
\module{\frac{2n+1}{2\lambda}}{2m}{p},& \qquad n\in \{0,...,\frac{p-3}{2}\}, m \in \{0,...,\frac{p-3}{2}\}
\end{align*}
is a minimal generating set for $\mathcal{G}^{\text{ss}}(\mathcal{C}^{0})$. Notice that the product of any pair of elements of this minimal generating set will be a sum of generators of the form $\module{\frac{x}{2\lambda}}{i}{p\ell}$ where $i$ must be even, and hence will be a sum of elements in the minimal generating set with positive coefficients. Therefore, this minimal generating set in fact gives a $\mathbb{Z}_+$-basis for the ring when $p$ is odd.\\

When $p$ is even, by the exact same considerations we see that
\begin{align*}
\module{\frac{2n}{2\lambda}}{2m}{0}, & \qquad n \in \{0,1,\dots,\frac{p-2}{2}\}, m\in\{0,1,\dots,\frac{p-2}{2}\}\\
\module{\frac{2n+1}{2\lambda}}{2m+1}{0},& \qquad n \in \{0,1,\dots,\frac{p-2}{2}\}, m \in \{0,1,\dots,\frac{p-4}{2}\}
\end{align*}
is a minimal generating set. Further, the product of generators of any of these generators with another of the same form will be a sum of generators  $\module{\frac{2n}{2\lambda}}{2m}{0}$ with positive coefficients. A product of a generator $\module{\frac{2n}{2\lambda}}{2m}{0}$ with another $\module{\frac{2n+1}{2\lambda}}{2m+1}{0}$ will then be a sum of generators $\module{\frac{2n+1}{2\lambda}}{2m+1}{0}$ with positive coefficients. Hence, this minimal generating set is again a $\mathbb{Z}_+$-basis.\\
\end{proof}

\subsection{Comparison for odd $p$} \label{oddp}
Let $V,W \in \UHbar{2}$-Mod and $w \in W$ a highest weight vector of weight $\lambda$. For any $\gamma \in \mathbb{C}$, define $\Psi_{\gamma}:\mathbb{Z}[z] \rightarrow \mathbb{C}$ by $\Psi_{\gamma}(z^s)=q^{\gamma s}$. Then, as noted in the proof of Lemma 6.6 in \cite{CGP}, the open Hopf links $\Phi_{V,W}$ satisfy $\Phi_{V,W}(w)=\Psi_{\lambda+1-p}(\chi(V))w$ where $\chi(V)$ is the character of $V$. Clearly then, when $W$ is simple, $\Phi_{V,W}=\Psi_{ \lambda+1-p}( \chi (V))\Id_W$. Using the fact that the Hopf link $\sS^{\hopflink}_{V,W}$ is the trace of the open Hopf link $\Phi_{V,W}$ and the appropriate character formulas \cite[Equation (16)]{CGP}, it is easy to show that 
\begin{equation}
\sS^{\hopflink}_{S_i \otimes \mathbb{C}_{pk}^H,S_j \otimes \mathbb{C}_{p \ell}^H}=(-1)^{k(j+p(\ell-1)+1)+(i+1)(\ell+1)+\ell+1}\frac{\{(i+1)(j+1)\}}{\{j+1\}} \tr(\Id_{S_j \otimes \mathbb{C}_{p\ell}^H})
\end{equation}
 where $\text{tr}(\Id_{S_j \otimes \mathbb{C}_{p\ell}^H})= (-1)^{(1-p)\ell+j}[j+1].$ The Hopf links in $\mathcal{H}_{i\mathbb{R}^{\oplus}}$ (see Subsection \ref{Heisenberg}) are easily seen to be $\sS^{\hopflink}_{\FockH{\gamma_1,\FockH{\gamma_2}}}=e^{\pi i \gamma_1\gamma_2}e^{\pi i \gamma_2\gamma_1}=e^{2\pi i\gamma_1\gamma_2}$. Therefore, the Hopf links are given by
\begin{align}\label{atypicaluqhopf}
\hopf{\gamma_1}{i}{k}{\gamma_2}{j}{\ell}&=\sS^{\hopflink}_{S_i \otimes \mathbb{C}_{pk}^H, S_j \otimes \mathbb{C}_{p\ell}^H}\cdot \sS^{\hopflink}_{\FockH{\gamma_1},\FockH{\gamma_2}} \nonumber \\
&=(-1)^{k(j+p(\ell-1)+1)+(i+1)(\ell+1)+\ell+1+(1-p)\ell+j}e^{2\pi i \gamma_1 \gamma_2}[(i+1)(j+1)] \nonumber\\
&=(-1)^{(i+1)(\ell+1)+(j+1)(k+1)+p(k\ell+k+\ell)}e^{2 \pi i \gamma_1 \gamma_2}[(i+1)(j+1)].
\end{align}

Normalizing the $\sS^{\hopflink}$ matrix and restricting to odd $p$ (i.e., using $i,j \in 2\mathbb{Z}$ for all elements of our generating set) gives
\begin{align*}
\frac{\hopf{\gamma_1}{i}{k}{\gamma_2}{j}{\ell}}{\hopf{0}{0}{0}{\gamma_2}{j}{\ell}}=e^{\pi i( 2\gamma_1\gamma_2+pkl)}\frac{\{(i+1)(j+1)\}}{\{j+1\}}.
\end{align*} 

Below, in order to find $\sS^\chi$, we shall use certain modular transformation properties from \cite{C}. For this, it will be beneficial for us to re-parametrize by setting $(s,s')=(i+1,-2\lambda_p\gamma_1-pk), (n,n')=(j+1,-2\lambda_p\gamma_2-p\ell)$ and adopting the notation $\sS^{\hopflink}_{(s,s'),(n,n')}:=\hopf{\gamma_1}{i}{k}{\gamma_2}{j}{\ell}$, we see that
\begin{equation}\label{atypicalinfty}
\frac{\sS^{\hopflink}_{(s,s'),(n,n')}}{\sS^{\hopflink}_{(1,0),(n,n')}}= q^{-n's'}\frac{\{ns\}}{\{n\}}. 
\end{equation}

 Recall that $\ind( \FockH{\gamma} \boxtimes (S_i \otimes \mathbb{C}_{p \ell}^H))=\bigoplus\limits_{k \in \mathbb{Z}} \FockH{\gamma+ k\lambda_p} \boxtimes (S_i \otimes \mathbb{C}_{p(k + \ell)}^H)$ which corresponds to 
\[ \bigoplus\limits_{k \in \mathbb{Z}} \FockH{\gamma+k\lambda_p} \boxtimes M_{1-(k+\ell),i+1} \]
 under the correspondence in Proposition \ref{correspondence}. This module is the spectral flow (\cite[Subsection 3.1.2]{C})
\begin{equation}\label{sigmass} \sigma^{s'}(W_s):=\bigoplus\limits_{\tilde{k} \in \mathbb{Z}} \FockH{\tilde{\lambda_p}(\tilde{k}-\frac{s'}{p})} \boxtimes M_{1+\tilde{k},s}' \end{equation}
of the module $W_s$ defined in \cite[Subsection 4.2]{C}, where $\tilde{\lambda_p}=-\lambda_p$, $\tilde{k}=-k-\ell$, and $(s,s')=(i+1,-2\lambda_p \gamma-p\ell)$. Consider the basis for odd $p$ given in Proposition $\ref{Grothendieck}$. By applying the relation $\module{\gamma}{p-1-\alpha}{p\ell}=-\module{\gamma}{\alpha-1}{p(\ell+1)}$ to the generators of the form $\module{\frac{2n+1}{2\lambda_p}}{2m}{p}$, keeping in mind that $\tilde{\lambda_p}=-\lambda_p$ and that the correspondence in Proposition \ref{correspondence} preserves tensor structure up to character, it is easy to see that the Verlinde algebra of characters, $\mathcal{V}(\Bp)$, generated by the atypical modules of $\Bp$ has a generating set 
\[ \{ \mathrm{ch}[\sigma^{s'}(W_s)]\, |\, (s,s') \in \Lambda_p \} \]
where \begin{equation}\label{Lambdap}
 \Lambda_p:= \{(s,s')|0 < s\leq p-1, \; 0 \leq s' \leq p-1,\; s+s'+1 \in 2\mathbb{Z}\}.
\end{equation}
This set is closed under modular transformations, and we will show that the corresponding $\sS^{\chi}$-matrix agrees with the $\sS^{\hopflink}$-matrix \eqref{atypicalinfty} up to normalised conjugation. From the character formula for $W_s$ and the relation $\text{ch}[\sigma^{s'}(M)](u;\tau)=q^{\frac{s'^2}{4p}-\frac{s'^2}{2}}x^{\frac{s'}{p}-2s'}\text{ch}[M](u+\tau\frac{s'}{2};\tau)$ (see \cite[Subsections 4.3 and 3.1.2]{C}), we see that the character of $\sigma^{s'}(W_s)$ is given by
\begin{equation} \label{sigmawchar}\text{ch}[\sigma^{s'}(W_s)](u;\tau)=\frac{q^{\frac{s'^2}{4p}-\frac{s'^2}{2}}x^{\frac{s'}{p}-2s'}}{\eta(\tau)^2}\sum\limits_{n \in \mathbb{Z}}\left( \frac{q^{p(n+\frac{1}{2}-\frac{s}{2p})^2}}{1-xq^{p(n+\frac{1}{2}-\frac{s}{2p})+\frac{s'}{2}}}-\frac{q^{p(n+\frac{1}{2}+\frac{s}{2p})^2}}{1-xq^{p(n+\frac{1}{2}+\frac{s}{2p})+\frac{s'}{2}}} \right) 
\end{equation}
where $x=e^{2\pi i u}, q=e^{2 \pi i \tau}$. 
Even though $s=0, s=p$ are not allowed in \eqref{Lambdap}, it is easy to see that
with $s=0,p$ in equation \eqref{sigmawchar}, for all $s'$, we have
\begin{align}
\text{ch}[\sigma^{s'}(W_0)](u;\tau)&=\text{ch}[\sigma^{s'}(W_{-p})](u;\tau)=0,\\
\text{ch}[\sigma^{s'}(W_s)](u;\tau)&+\text{ch}[\sigma^{s'}(W_{-s})](u;\tau)=0.
\end{align}
Recall the notation from \cite{C}:
\begin{align}
\Pi(v;\tau)=q^{1/6}(z-z^{-1})\prod\limits_{n=1}^{\infty}(1-z^2q^n)(1-q^n)^2(1-z^{-2}q^n)
\end{align}
with $z=e^{2\pi iv}$.
By \cite[Theorem 3.6 and Subsection 4.3]{C}, we have the following:
\begin{align}
\text{ch}[\sigma^{s'}(W_s)](\frac{u}{\tau};\frac{-1}{\tau})&=\lim\limits_{v \to 0} \frac{\Pi(v;\tau)}{\eta(\tau)^2}\text{ch}[\sigma^{s,s'}\chi_p](\frac{u}{\tau};\frac{v}{\tau};\frac{-1}{\tau})\nonumber\\
&=\lim\limits_{v \to 0} \frac{\Pi(v;\tau)}{\eta(\tau)^2}e^{\frac{2 \pi i}{\tau}(kv^2-\frac{u^2}{p})}\sum\limits_{(n,n')\in S_p} S_{(s,s'),(n,n')}\text{ch}[\sigma^{(n,n')}(\chi_p)](u;v;\tau)\nonumber\\
&=\sum\limits_{(n,n')\in S_p}e^{\frac{-2 \pi i u^2}{p \tau}} S_{(s,s'),(n,n')}\lim\limits_{v \to 0} \frac{\Pi(v;\tau)}{\eta(\tau)^2}\text{ch}[\sigma^{(n,n')}(\chi_p)](u;v;\tau)\nonumber\\
&= \sum\limits_{(n,n')\in S_p}e^{\frac{-2 \pi i u^2}{p \tau}} S_{(s,s'),(n,n')} \text{ch}[\sigma^{n'}(W_n)](u,\tau),
\end{align}
with $S_{(s,s'),(n,n')}=\frac{-1}{p}e^{-\frac{2\pi i}{2p}(sn-s'n')}$ and $S_p:=\{(n,n')|-p\leq n \leq p-1, 0 \leq n' \leq p-1, n+n'+1 \in 2\mathbb{Z}\}$. It follows that the character of $\sigma^{s'}(W_s)$ satisfies
\begin{align}
\text{ch}[\sigma^{s'}(W_s)](\frac{u}{\tau};\frac{-1}{\tau})= \sum\limits_{(n,n')\in \Lambda_p}e^{\frac{-2 \pi i u^2}{p \tau}} \left(S_{(s,s'),(n,n')}-S_{(s,s'),(-n,n')}\right) \text{ch}[\sigma^{n'}(W_n)](u,\tau),
\end{align}
and therefore,
\begin{align}
S^{\chi}_{(s,s'),(n,n')}&=e^{\frac{-2 \pi i u^2}{p \tau}}\left(S_{(s,s'),(n,n')}-S_{(s,s'),(-n,n')}\right) \nonumber \\
&=-\frac{1}{p} e^{\frac{-2 \pi i u^2}{p \tau}}e^{\frac{2\pi i}{2p}n's'}\left(e^{-\frac{2\pi i}{p}ns}-e^{\frac{2 \pi i}{p}ns}\right)  \nonumber \\
&=-\frac{1}{p} e^{\frac{-2 \pi i u^2}{p \tau}}q^{n's'}\{ns\}. \label{atypicalchinonnormal}
\end{align}

The unit object in $\C$ is $\FSC{0}{0}{0}$ which induces to $\bigoplus\limits_{k \in \ZZ} \FSC{k\lambda}{0}{pk}$, so the identity object corresponds to $\sigma^0(W_1)$, which we will simply denote by $\mathds{1}$ in the index. Hence, the normalized $\sS^{\chi}$-matrix is given by
\begin{align}\label{atypicalschi}
\frac{\sS^{\chi}_{(s,s'),(n,n')}}{\sS^{\chi}_{\mathds{1},(n,n')}}=q^{n's'}\frac{[ns]}{[n]}.
\end{align}

The following proposition follows by comparing \eqref{atypicalschi} and \eqref{atypicalinfty}:

\begin{proposition}\label{atypicalhopf}
For atypical modules, the matrices $\mathsf{S}^{\hopflink}$ and $\mathsf{S}^{\chi}$ are in agreement up to normalized conjugation. That is,
\begin{align}
\frac{\sS^{\chi \ast}_{(s,s'),(n,n')}}{\sS^{\chi \ast}_{\mathds{1},(n,n')}}= \frac{\sS^{\hopflink}_{(s,s'),(n,n')}}{\sS^{\hopflink}_{\mathds{1},(n,n')}}
\end{align}
for $(s,s'),(n,n')\in\Lambda_p$.
\end{proposition}

It then follows that the matrix $\sS^{\hopflink}$ is invertible since $\sS^{\chi}$ is, and by the standard argument the categorical Verlinde formula holds. We include this argument here for completeness:

Recall that the Hopf links give a one dimensional representation of the fusion ring and therefore satisfy
\[
\frac{\sS^{\hopflink}_{(s,s'),(n,n')}}{\sS^{\hopflink}_{\mathds{1},(n,n')}}\frac{\sS^{\hopflink}_{(t,t'),(n,n')}}{\sS^{\hopflink}_{\mathds{1},(n,n')}}=\sum\limits_{(m,m')} N^{(m,m')}_{(s,s'),(t,t')} \frac{\sS^{\hopflink}_{(m,m'),(n,n')}}{\sS^{\hopflink}_{\mathds{1},(n,n')}} \]
where the sum runs over the pairs corresponding to the basis of $\mathcal{G}^{\text{ss}}(\mathcal{C}^{0})$ in Proposition \ref{Grothendieck}. The matrix $\sS^{\chi}$ is invertible, so by the above corollary, $\sS^{\hopflink}$ is also invertible. Cancelling out the $\sS^{\hopflink}_{(1,0),(n,n')}$ terms and multiplying by $(\sS^{\hopflink})^{-1}_{(k,k'),(n,n')}$ for any fixed $(k,k')$ yields
\[
\frac{\sS^{\hopflink}_{(s,s'),(n,n')}\sS^{\hopflink}_{(t,t'),(n,n')}}{\sS^{\hopflink}_{\mathds{1},(n,n')}}(\sS^{\hopflink})^{-1}_{(n,n'),(k,k')}=\sum\limits_{(m,m')} N^{(m,m')}_{(s,s'),(t,t')} \sS^{\hopflink}_{(m,m'),(n,n')}(\sS^{\hopflink})^{-1}_{(n,n'),(k,k')}. \]
Summing over the index $(n,n')$ then gives
\begin{align}
\sum\limits_{(n,n')} \frac{\sS^{\hopflink}_{(s,s'),(n,n')}\sS^{\hopflink}_{(t,t'),(n,n')}(\sS^{\hopflink})^{-1}_{(n,n'),(k,k')}}{\sS^{\hopflink}_{\mathds{1},(n,n')}} &=\sum\limits_{(m,m')} N^{(m,m')}_{(s,s'),(t,t')} \left(\sS^{\hopflink}(\sS^{\hopflink})^{-1} \right)_{(m,m'),(k,k')} \nonumber\\
&=\sum\limits_{(m,m')} N^{(m,m')}_{(s,s'),(t,t')}\delta_{(m,m'),(k,k')}\nonumber\\
&=N^{(k,k')}_{(s,s'),(t,t')}
\end{align}
which is the Verlinde formula and by Proposition \ref{atypicalhopf}, the Verlinde formula also holds for atypical $\Bp$ modules:

\begin{corollary} \label{Verlindeodd} The Verlinde formula holds for the Verlinde algebra of characters generated by atypical modules of $\Bp$ when $p$ is odd. That is,
\[\sum\limits_{(n,n') \in \Lambda_p} \frac{\sS^{\hopflink}_{(s,s'),(n,n')}\sS^{\hopflink}_{(t,t'),(n,n')}(\sS^{\hopflink})^{-1}_{(n,n'),(k,k')}}{\sS^{\hopflink}_{\mathds{1},(n,n')}}=N^{(k,k')}_{(s,s'),(t,t')}\]
where $\Lambda_p$ is given in \eqref{Lambdap}.
\end{corollary}

\subsection{Comparison for even $p$}

When $p$ is even, $\Bp$ has half integer $L(0)$-grading and $\mathbb{Z}_2$-grading given by
\[ \Bp=\Bp^{\overline{0}} \oplus \Bp^{\overline{1}}, \]
where $\Bp^{\overline{0}}$ is the integer part of the $\frac{1}{2}\mathbb{Z}$-grading and $\Bp^{\overline{1}}$ is the non-integer part. In this case, we instead compare the ring $\mathcal{G}^{ss}(\mathcal{C}^{0}_{\overline{0}})$ for the modules which lift to $\Bp^{\overline{0}}$ modules with the Verlinde algebra of characters of $\Bp^{\overline{0}}$. Note that when we computed $\mathcal{G}^{ss}(\mathcal{C}^{0})$ we used the relation $\module{\gamma}{i}{\ell p}=\module{\gamma+\lambda}{i}{(\ell+1)p}$ which no longer holds. To remedy this, we need only notice that the objects $\module{\gamma}{i}{\ell p}$ and $\module{\gamma+2\lambda}{i}{(\ell+2)p}$ induce to the same $\Bp^{\overline{0}}$ module and are therefore equal in $\mathcal{G}^{ss}(\mathcal{C}_{\overline{0}}^{0})$. Hence, it is enough to add to the basis for $\mathcal{G}^{ss}(\mathcal{C}^{0})$ with $p$ even obtained in Proposition \ref{Grothendieck} the shifted elements $\module{\gamma+\lambda}{i}{(\ell+1)p}$ for each basis element $\module{\gamma}{i}{\ell p}$. Hence, a basis for $\mathcal{G}^{ss}(\mathcal{C}_{\overline{0}}^{0})$ is given by
\begin{align*}
\module{\frac{2n}{2\lambda}}{2m}{0}, \module{\frac{2n-p}{2\lambda}}{2m}{p}& \qquad n \in \{0,1,\dots,\frac{p-2}{2}\}, m\in\{0,1,\dots,\frac{p-2}{2}\},\\
\module{\frac{2n+1}{2\lambda}}{2m+1}{0}, \module{\frac{2n+1-p}{2\lambda}}{2m+1}{p}& \qquad n \in \{0,1,\dots,\frac{p-2}{2}\}, m \in \{0,1,\dots,\frac{p-4}{2}\}.
\end{align*}
This is not a $\mathbb{Z}_+$ basis with respect to the product 
\begin{equation}\label{Grothendieckproduct}\module{\frac{x}{2\lambda}}{i}{p\ell} \mdot \module{\frac{y}{2\lambda}}{j}{p\ell '}\left\{=\begin{array}{cc}
\sum\limits_{\substack{l=|i-j| \\ \text{by } 2}}^{i+j} \module{\frac{x+y \text{ mod 2p}}{2\lambda}}{l}{p(\ell+\ell')\text{ mod 2p}}  & \text{ if $i+j<p$},\\
&\\
\sum\limits_{\substack{l=|i-j| \\ \text{by } 2}}^{2p-4-i-j}\module{\frac{x+y \text{ mod 2p}}{2\lambda}}{l}{p(\ell+\ell')\text{ mod 2p}}  & \text{if $i+j \geq p$}. \\
\end{array}\right.\\
\end{equation}
Applying the relation $\module{\gamma}{j-1}{p(\ell+1)}=-\module{\gamma}{p-1-j}{p\ell}$, we obtain a basis
\begin{align*}
\module{\frac{2n}{2\lambda}}{2m}{0}, \module{\frac{2n-p}{2\lambda}}{p-2-2m}{0}& \qquad n \in \{0,1,\dots,\frac{p-2}{2}\}, m\in\{0,1,\dots,\frac{p-2}{2}\},\\
\module{\frac{2n+1}{2\lambda}}{2m+1}{0}, \module{\frac{2n+1-p}{2\lambda}}{p-2-(2m+1)}{0}& \qquad n \in \{0,1,\dots,\frac{p-2}{2}\}, m \in \{0,1,\dots,\frac{p-4}{2}\},
\end{align*}
which is easily seen to be a $\mathbb{Z}_+$-basis. We have already seen that the Hopf links satisfy equation \eqref{atypicaluqhopf}. Restricting this equation to $p$ even ($k=\ell=0$) gives 
\begin{equation}
(\sS^{\hopflink}_{\Bp^{\overline{0}}})_{(s,s'),(n,n')}=(-1)^{s+n}q^{-n's'}\{ns\},
\end{equation}
once again making the substitution $(s,s')=(i+1,-2\lambda_p\gamma_1-pk), (n,n')=(j+1,-2\lambda_p\gamma_2-p\ell)$ and the indices run over the set $\Lambda_p \cup \tilde \Lambda_p$ where $\tilde \Lambda_p:=\{ (s,s'-p)| (s,s') \in \Lambda_p\}$. We therefore have four cases:
{\renewcommand{\arraystretch}{1.5}
\begin{center}
\begin{tabular}{ |c|c| }
\hline
$\left( (s,s'),(n,n') \right)$ & $(\sS^{\hopflink}_{\Bp^{\overline{0}}})_{(s,s'),(n,n')}$\\
\hline 
 $\Lambda_p \times \Lambda_p$ & $(-1)^{s+n}q^{-n's'}\{ns\}$ \\ \hline
 $\Lambda_p \times \tilde{ \Lambda}_p$ & $(-1)^{n+1}q^{-n's'}\{ns\}$ \\ \hline
 $\tilde{\Lambda}_p \times \Lambda_p$ & $(-1)^{s+1}q^{-n's'} \{ns\}$ \\ \hline
$\tilde{\Lambda}_p \times \tilde{ \Lambda}_p$ & $q^{-n's'}\{ns\}$ \\ \hline
\end{tabular}
\end{center}}

Following the analysis in Subsection \ref{oddp} and replacing characters with supercharacters where necessary, we see that the $\sS^{\chi}$ matrix for the atypical $\Bp$ modules $\{ \sigma^{s'}(W_s)| (s,s' ) \in \Lambda_p\}$ are still given by equation \eqref{atypicalchinonnormal}. However, each $\Bp$ module $\sigma^{s'}(W_s)$ splits into two $\Bp^{\overline{0}}$ modules $\sigma^{s'}(W_s)^{\overline{0}}$ and $\sigma^{s'}(W_s)^{\overline{1}}$, where we recall that $\sigma^{s'}(W_s)$ is given by equation \eqref{sigmass} and $\sigma^{s'}(W_s)^{\overline{0}},\sigma^{s'}(W_s)^{\overline{1}}$ are given by the even and odd summands respectively. It then follows from equation \eqref{atypicalchinonnormal} and \cite[Equation 4.9]{C} that these matrices agree up to normalised conjugation.

\section{The characters of $\Bp$ and of a QH-reduction} \label{SectionQH}

In this section, we set $n = p-1$ and compare the character of $\Bp$ with a certain Quantum-Hamiltonian reduction of $\sln{n}$ as announced in \cite[Remark 5.6]{C}. For more details about Quantum-Hamiltonian reduction, see \cite{Ara} for instance. 

Recall that $\sln{n-1}$ can be embedded in $\sln{n}$ such that
\[ \sln{n} \cong \sln{n-1} \oplus \rho_{n-1} \oplus \overline{\rho_{n-1}} \oplus \mathbb{C} \]
where $\rho_{n-1}$ is the standard representation of $\sln{n-1}$ and $\overline{\rho_{n-1}}$ its conjugate. To carry computations, we see $\sln{n-1}$ as embedded in the upper-left square of the matrix realisation of $\sln{n}$. Fix the $\sln{2}$-triplet $\{F,H,E\}$ in $\sln{n-1} $ as follows
\begin{align} \label{eq:sl2tripleQHcharBp}
F:&=\sum\limits_{i=1}^{n-2}e_{i+1,i}, & H:&=\frac{1}{2}\sum\limits_{i=1}^{n-1}(n-2i)e_{ii}, & E:&=\sum\limits_{i=1}^{n-2}e_{i,i+1} 
\end{align}

Notice that for any matrix $h=\sum\limits_{j=1}^n \lambda_j e_{jj} \in \Hlie$, we have
\[ [F,h]=\sum\limits_{i=1}^{n-2}(\lambda_i- \lambda_{i+1})e_{i+1,i}, \]
and the matrix must be of trace zero, so $0=(n-1)\lambda_1 + \lambda_n$. It follows that the subalgebra of $\Hlie$ annihilated by $F$ under the adjoint action is $\text{Span}_{\mathbb{C}} \{A\}$ where $A$ is the diagonal matrix 
$A:=\mathrm{diag}\{\frac{1}{n},...,\frac{1}{n},\frac{1-n}{n}\}$. 
Let $K=Av$ for complex $v$ and set $x=e^v$. Let $\Delta_+$ denote the set of positive roots of the full $\sln{n}$, $\Delta_+^0 \subset \Delta_+$ the set of positive roots such that $\alpha(H)=0$, and $\Delta_+^{\frac{1}{2}} \subset \Delta_+$ the positive roots  for which $\beta(H)=1/2$. By \cite[Equation (11)]{KW}, the character of the Quantum-Hamiltonian reduction associated to \eqref{eq:sl2tripleQHcharBp} is given by the following formula:
\begin{equation}\label{eq:KW} (-i)^{\frac{p(p+1)}{2}}q^{\frac{(p^2-1)(p-2)}{24}}\frac{\eta(p \tau)^{-\frac{1}{2}p^2+\frac{5}{2}p-3}}
	{\eta(\tau)^{p-3}} \frac{ \prod\limits_{\alpha \in \Delta_+} \vartheta_{11}(p \tau, \alpha(K- \tau H) )}{\left( \prod\limits_{\alpha \in \Delta_+^0} \vartheta_{11}(\tau,\alpha(K)) \right)\left( \prod\limits_{\beta \in \Delta^{\frac{1}{2}}} \vartheta_{01}(\tau,\beta(K)) \right)^{1/2} } 
\end{equation}

where $\vartheta_{11}$ and $\vartheta_{01}$ are the standard Jacobi theta functions:
\begin{align}
\vartheta_{11}(\tau,z)&=-iq^{1/12}u^{-1/2}\eta(\tau) \prod\limits_{k=1}^{\infty}(1-u^{-1}q^k)(1-uq^{k-1}) \; ,\\
\vartheta_{01}(\tau,z)&=\prod\limits_{k=1}^{\infty}(1-u^{-1}q^{k-1/2})(1-q^k)(1-uq^{k-1/2})
\end{align}
where $u=e^{2\pi i z}$ and $q=e^{2 \pi i \tau}$. 

\begin{theorem}\label{BpQH} Let $p \in \ZZ$. Then the character of $\Bp$ is given by equation~\eqref{eq:KW}.
\end{theorem}
\begin{proof} As for proving~\cite[Theorem 5.5]{C}, a character is viewed as a formal power series where we admit $\frac{1}{1-x} = \sum\limits_{r=0}^\infty x^r$. Two formal power series $X, Y$ will be seen as equivalent if $X = \gamma q^a x^b Y$ for $\gamma \in \CC$ and $a,b \in \QQ$. Also here, it is enough to show that $\ch{\Bp}{x;q}$ is equivalent to \eqref{eq:KW} since both are formal power series of the form $q^{-\frac{c_p}{24}}(1 + \cdots)$.\\

Let's then make the formula \eqref{eq:KW} explicit, keeping in mind that $n = p-1$ here. First let's fix a standard choice of positive roots for $\sln{n} = \sln{p-1}$ and compute $\Delta_+^0$ and $\Delta^{\frac{1}{2}}$:
\begin{itemize}
\item $\Delta_+ = \Delta_+(\sln{n}) = \left\{\alpha_{i,j} =\sum\limits_{\ell=i}^j \alpha_\ell \; | \; 1\leq i \leq j \leq n-1\right\}$ where $\alpha_\ell(e_{ii}) = \delta_{\ell,i} - \delta_{\ell+1,i}$;
\item $\alpha_{i,j}(K)=0$ if $j<n-1$;
\item $\alpha_{i,n-1}(K)=v$;
\item $\alpha_{i,j}(H)=j-i+1$ if $j<n-1$;
\item $\alpha_{i,n-1}(H)=\frac{n}{2}-i$;
\item $\Delta_+^0 = \{\alpha_{\frac{n}{2},n-1}\}$ and $\Delta^{\frac{1}{2}} = \emptyset$ if $n$ is even ($p$ odd);
\item $\Delta_+^0 = \emptyset$ and $\Delta^{\frac{1}{2}} = \{\alpha_{\frac{n-1}{2},n-1},-\alpha_{\frac{n+1}{2},n-1}\}$ if $n$ is odd ($p$ even).
\end{itemize}

We shall now examine each factor in \eqref{eq:KW}. For any choice of $n$, since $\alpha_{i,j}(K-\tau H)=\tau(i-j-1)$ if $i<j$ and $\alpha_{i,n-1}(K-\tau H)=v-\tau(\frac{n}{2}-i)$ one obtains:
\begin{align*}
\frac{\vartheta_{11}(p \tau, \alpha_{i,j}(K-\tau H))}{\eta(p \tau)} &\sim \prod\limits_{k=1}^{\infty} (1-q^{p(k-1)+(j-i+1)})(1-q^{pk-(j-i+1)}) \qquad \qquad (j<n-1)\\
\frac{\vartheta_{11}(p \tau, \alpha_{i,n-1}(K-\tau H))}{\eta(p \tau)} &\sim \prod\limits_{k=1}^{\infty} (1-x^{-1}q^{p(k-1)+(\frac{n}{2}-i)})(1-xq^{pk-(\frac{n}{2}-i)}),
\end{align*}
where we used the observation $\prod\limits_{k=1}^\infty (1-u^{-1}q^k)(1-uq^{k-1}) \sim \prod\limits_{k=1}^\infty (1-uq^k)(1-u^{-1}q^{k-1})$ on each line. Then, aiming to find the product over $\Delta_+$ as in the numerator of \eqref{eq:KW}, we write:
\begin{align}
&\prod_{k=1}^\infty \prod_{1\leq i \leq j < n-1} \frac{\vartheta_{11}(p \tau, \alpha_{i,j}(K-\tau H))}{\eta(p \tau)} 
\sim \prod\limits_{k=1}^{\infty} \prod_{1\leq i \leq j < n-1}  (1-q^{pk+(j-i+1-p)})(1-q^{pk-(j-i+1)}) \nonumber\\
&\sim \prod\limits_{k=1}^{\infty} \prod_{1\leq a \leq n-2}  (1-q^{pk-(p-a)})^{p-2-a}(1-q^{pk-a})^{p-2-a} \nonumber\\
&\sim \prod\limits_{k=1}^{\infty} (1-q^{pk-1})(1-q^{p(k-1)+1}) \prod_{1\leq b \leq n} (1-q^{pk-b})^{p-4} \nonumber\\ 
&\sim \prod\limits_{k=1}^{\infty} (1-q^{pk-1})(1-q^{p(k-1)+1}) \prod_{1\leq b \leq n} (1-q^{pk-b})^{p-4} \nonumber\\
&\sim \prod\limits_{k=1}^{\infty} (1-q^{pk-1})(1-q^{p(k-1)+1}) \left(\frac{1-q^k}{1-q^{pk}}\right)^{p-4}\nonumber\\
&\sim \left(\frac{\eta(\tau)}{\eta(p\tau)}\right)^{p-4} \cdot \prod\limits_{k=1}^{\infty} (1-q^{pk-1})(1-q^{p(k-1)+1}).
\end{align}
We also will need the following for the numerator:
\begin{align}
\prod\limits_{1\leq i \leq n-1}&\frac{\vartheta_{11}(p \tau, \alpha_{i,n-1}(K-\tau H))}{\eta(p \tau)} \sim \prod\limits_{k=1}^{\infty} \prod\limits_{1\leq i \leq n-1} (1-x^{-1}q^{p(k-1)+(\frac{n}{2}-i)})(1-xq^{pk-(\frac{n}{2}-i)})\nonumber\\
&=\prod\limits_{k=0}^{\infty} \frac{(1-xq^{k+\frac{p}{2}+\frac{3}{2}})(1-x^{-1}q^{k-\frac{p}{2}+\frac{3}{2}})}{(1-xq^{pk+\frac{3}{2}p-\frac{1}{2}})(1-xq^{pk+\frac{3}{2}p+\frac{1}{2}})(1-x^{-1}q^{p(k+\frac{1}{2})-\frac{1}{2}})(1-x^{-1}q^{p(k+\frac{1}{2})+\frac{1}{2}})}.
\end{align}
Finally completing the product of $\vartheta_{11}$s over $\Delta_+$ in the numerator of \eqref{eq:KW} yields: 
 \begin{align}
 \prod\limits_{\alpha \in \Delta_+} &\vartheta_{11}(p \tau, \alpha(K - \tau H) ) \sim \eta(p\tau)^{\frac{(n-2)(n-1)}{2}+(n-1)-(n-3)}\eta(\tau)^{n-3} \cdot \prod\limits_{k=1}^{\infty} (1-q^{pk-1})(1-q^{p(k-1)+1}) \nonumber\\
 &\times\prod\limits_{k=0}^{\infty} \frac{(1-xq^{k+\frac{p}{2}+\frac{3}{2}})(1-x^{-1}q^{k-\frac{p}{2}+\frac{3}{2}})}{(1-xq^{pk+\frac{3}{2}p-\frac{1}{2}})(1-xq^{pk+\frac{3}{2}p+\frac{1}{2}})(1-x^{-1}q^{p(k+\frac{1}{2})-\frac{1}{2}})(1-x^{-1}q^{p(k+\frac{1}{2})+\frac{1}{2}})}.
 \end{align}
The $\eta$ factors written in terms of $p$ instead of $n$ read:
\begin{align}
\eta(p\tau)^{\frac{(p-3)(p-2)}{2}-(p-4)+p-2}\eta(\tau)^{p-4}
=\eta(p\tau)^{\frac{p^2}{2}-\frac{5p}{2}+5}\eta(\tau)^{p-4}.
\end{align}
The denominator of \eqref{eq:KW} is given up to equivalence by
\begin{align}
&\eta(\tau)\prod\limits_{k=0}^{\infty}(1-x^{-1}q^{k+1})(1-x^nq^k) & &p\mathrm{\; odd,}\\
&\eta(\tau)\prod\limits_{k=0}^{\infty}(1-x^{-1}q^{k+\frac{1}{2}})(1-xq^{k+\frac{1}{2}}) & &p \mathrm{\; even.}
\end{align}
In effect, collecting various $\eta$ factors arising from $\vartheta$s now gets us:
$\eta(p\tau)^{\frac{p^2}{2}-\frac{5p}{2}+5}\eta(\tau)^{p-5}$.
Canceling these with the $\eta$ factors already present in the \eqref{eq:KW},
we are left with simply $\dfrac{\eta(p\tau)^2}{\eta(\tau)^2}$.

We now show the rest of the calculation for $p$ be odd, the other case is similar.
In this case, the character, up to $\sim$ equivalence has now simplified to:
\begin{align}
&\dfrac{\eta(p\tau)^2}{\eta(\tau)^2}\prod\limits_{k=0}^{\infty} \frac{(1-q^{p(k+1)-1})(1-q^{pk+1})(1-xq^{k+\frac{p}{2}+\frac{3}{2}})(1-x^{-1}q^{k-\frac{p}{2}+\frac{3}{2}})}{(1-xq^{pk+\frac{3}{2}p-\frac{1}{2}})(1-xq^{pk+\frac{3}{2}p+\frac{1}{2}})(1-x^{-1}q^{p(k+\frac{1}{2})-\frac{1}{2}})(1-x^{-1}q^{p(k+\frac{1}{2})+\frac{1}{2}})(1-x^{-1}q^{k+1})(1-xq^k)}
\nonumber\\
&=\dfrac{\eta(p\tau)^2}{\eta(\tau)^2}\prod\limits_{k=0}^{\infty} \frac{(1-q^{p(k+1)-1})(1-q^{pk+1})}
{(1-xq^{pk+\frac{3}{2}p-\frac{1}{2}})(1-xq^{pk+\frac{3}{2}p+\frac{1}{2}})(1-x^{-1}q^{p(k+\frac{1}{2})-\frac{1}{2}})(1-x^{-1}q^{p(k+\frac{1}{2})+\frac{1}{2}})}\nonumber\\
&\quad\quad\quad\times
\frac{\prod\limits_{k=-\frac{p}{2}+\frac{3}{2}}^{0}(1-x^{-1}q^k)}{\prod\limits_{k=0}^{\frac{p}{2}+\frac{1}{2}}{(1-xq^k)}}
\nonumber\\
&\sim\dfrac{\eta(p\tau)^2}{\eta(\tau)^2}\left(\prod\limits_{k=0}^{\infty} \frac{(1-q^{p(k+1)-1})(1-q^{pk+1})}
{(1-xq^{pk+\frac{3}{2}p-\frac{1}{2}})(1-xq^{pk+\frac{3}{2}p+\frac{1}{2}})(1-x^{-1}q^{p(k+\frac{1}{2})-\frac{1}{2}})(1-x^{-1}q^{p(k+\frac{1}{2})+\frac{1}{2}})}\right)\nonumber\\
&\quad\quad\times\frac{1}{(1-xq^{\frac{p}{2} -\frac{1}{2}})(1-xq^{\frac{p}{2} +\frac{1}{2}})}
\nonumber\\
&=\dfrac{\eta(p\tau)^2}{\eta(\tau)^2}\prod\limits_{k=0}^{\infty} \frac{(1-q^{p(k+1)-1})(1-q^{pk+1})}
{(1-xq^{pk+\frac{p}{2}-\frac{1}{2}})(1-xq^{pk+\frac{p}{2}+\frac{1}{2}})(1-x^{-1}q^{p(k+\frac{1}{2})-\frac{1}{2}})(1-x^{-1}q^{p(k+\frac{1}{2})+\frac{1}{2}})}\label{eq:finalKW}.
\end{align}
The character of $\Bp$ can be written in a product form using $\Bp=W_1$ and the \cite[Subsection 4.1, Proposition 5.2]{C} as follows:
\begin{align}
&\mathrm{ch}[\Bp](x;\tau)=\text{ch}[W_1](x;\tau)=
\lim_{z\rightarrow 1}\frac{\Pi(zq^{\frac{1}{2}};\tau)}{\eta(\tau)^2}q^{\frac{1}{4p}}\mathrm{ch}[\chi_p](x;zq^{\frac{1}{2}};\tau)\nonumber\\
&=\frac{1}{\eta(\tau)^2}q^{\frac{1}{4p}}q^{p/4-1/6}
\prod\limits_{k=0}^{\infty} \frac{(1-q^{pk+1})(1-q^{p(k+1)})^2(1-q^{p(k+1)-1})}
{(1-xq^{p(k+\frac{1}{2})+\frac{1}{2}})(1-xq^{p(k+\frac{1}{2})-\frac{1}{2}})(1-x^{-1}q^{p(k+\frac{1}{2})+\frac{1}{2}})(1-x^{-1}q^{p(k+\frac{1}{2})-\frac{1}{2}})}\nonumber\\
&\sim
\frac{\eta(p\tau)^2}{\eta(\tau)^2}
\prod\limits_{k=0}^{\infty} \frac{(1-q^{pk+1})(1-q^{p(k+1)-1})}
{(1-xq^{p(k+\frac{1}{2})+\frac{1}{2}})(1-xq^{p(k+\frac{1}{2})-\frac{1}{2}})(1-x^{-1}q^{p(k+\frac{1}{2})+\frac{1}{2}})(1-x^{-1}q^{p(k+\frac{1}{2})-\frac{1}{2}})}.
\end{align}
This matches with \eqref{eq:finalKW}, completing the proof.\\
\end{proof}

Theorem \ref{BpQH} shows that $\Bp$ is a quantum hamiltonian reduction up to character, as suggested in \cite[Remark 5.6]{C}.

\end{document}